\documentclass[12pt]{amsart}
\usepackage{graphicx,verbatim,amssymb,amsmath,multirow,bbold}
\usepackage[all]{xy}
\usepackage[active]{srcltx}
\usepackage{hyperref}

\DeclareSymbolFont{bbold}{U}{bbold}{m}{n}
\DeclareSymbolFontAlphabet{\mathbbold}{bbold}

\theoremstyle{definition}

\theoremstyle{remark}

\numberwithin{equation}{section}

\def\C{\mathbb{C}}
\def\D{\mathbb{D}}

\def\Q{\mathbb{Q}}
\def\R{\mathbb{R}}
\def\Sb{\mathbb{S}}
\def\T{\mathbb{T}}

\def\Z{\mathbb{Z}}
\def\1{\mathbbold{1}}

\def\cb{\mathbf{c}}

\def\Ac{\mathcal{A}}

\def\Fc{\mathcal{F}}
\def\Gc{\mathcal{G}}

\def\Lc{\mathcal{L}}

\def\Oc{\mathcal{O}}
\def\Pc{\mathcal{P}}

\def\Sc{\mathcal{S}}

\def\Uc{\mathcal{U}}
\def\Vc{\mathcal{V}}

\def\lf{\mathfrak{l}}

\def\RP{\mathbb{R}P}

\def\al{\alpha}

\def\ga{\gamma}

\def\la{\lambda}

\def\d{\partial}
\def\0{\varnothing}
\def\sm{\setminus}
\def\ol{\overline}

\def\le{\leqslant}
\def\ge{\geqslant}
\def\<{\langle}
\def\>{\rangle}

\def\ker{\mathrm{ker}}

\def\Hom{\mathrm{Hom}}
\def\vol{\mathrm{vol}}

\def\Len{\mathrm{Len}}

\def\area{\mathrm{area}}

\def\supp{\mathrm{supp}} 

\usepackage{faktor}

\def\St{\mathrm{St}}

\def\supp{\mathrm{supp}}

\def\Hs{\mathsf{H}}
\def\Gs{\mathsf{G}}

\def\Ts{\mathsf{T}}

\theoremstyle{Theorem}

\newtheorem*{lem*}{Lemma}
\newtheorem*{thm*}{Theorem}
\newtheorem*{cor*}{Corollary}

\begin{document}

\title[Valuations and arrangements]
{Valuations on polyhedra and topological arrangements}

\author[A.~Khovanski\u{i}]{Askold Khovanski\u{i}}

\author[V.~Kiritchenko]{Valentina Kiritchenko}

\author[V.~Timorin]{Vladlen Timorin}
\thanks{Results of the project ``Symmetry. Information. Chaos'', carried out within the framework of the Basic Research Program at HSE University,
 are presented in this work.}

\begin{abstract}
We revisit a classical theme of (general or translation invariant) valuations on convex polyhedra.
Our setting generalizes the classical one, in a ``dual'' direction to previously considered
 generalizations: while previous research was mostly concerned with variations of ground fields/rings,
 over which the vertices of polytopes are defined, we consider more general collections of defining hyperplanes.
No algebraic structures are imposed on these collections.
This setting allows us to uncover a close relationship between scissors congruence problems
 on the one hand and finite hyperplane arrangements on the other hand:
 there are many parallel results in these fields, for which the
 parallelism seems to have gone unnoticed.
In particular, certain properties of the Varchenko--Gelfand algebras for arrangements
 translate to properties of polytope rings or valuations.
Studying these properties is possible in a general topological setting,
 that is, in the context of the so-called topological arrangements, where hyperplanes
 do not have to be straight and may even have nontrivial topology.
\end{abstract}

\maketitle

\section{Historic overview}
\label{s:over}
This overview is aimed at non-experts; please refer to the next section for a brief description of results.

\textbf{Scissors congruence problems} go back at least to the times of Euclid and Archimedes; they have rich history.
Recall that two polygons in the plane are \emph{scissors congruent} if one can be cut into
 finitely many polygonal pieces, of which the other can be re-assembled.
Implicit in this definition is a certain group $\Gs$ of Euclidean isometries:
 only transformations from $\Gs$ are allowed when moving the pieces around.
Depending on the choice of $\Gs$, one obtains various scissors congruence problems
 called \emph{$\Gs$-scissors congruence problems}.
For example, when $\Gs$ is the full group of plane Euclidean isometries,
 two planar polygons are $\Gs$-scissors congruent if and only if they have the same area.
This result is very classical and is known as the \emph{Wallace--Bolyai--Gerwien theorem};
 the name acknowledges the independent explicit proofs of the early 19th century,
 while implicitly, the statement traces back to Euclid.
In fact, the same result holds true when $\Gs$ is replaced with the much smaller subgroup
 consisting only of parallel translations and half-turn rotations, as Hadwiger and Glur showed in \cite{HG51}.

Even in this simplest context, it is clear that certain abstract properties of the area play a decisive role,
 namely, additivity and $\Gs$-invariance.
Additivity property is formalized by the notion of a valuation.
Given a class $\Pc$ of convex polygons such that $\0\in\Pc$ and that is stable
 under taking intersections, define a \emph{valuation on $\Pc$} as a function $\mu$
 on $\Pc$ (with values, say, in $\R$ or in $\Q$) such that $\mu(\0)=0$
 and $\mu(P\cup Q)=\mu(P)+\mu(Q)-\mu(P\cap Q)$ whenever $P$, $Q$, and $P\cup Q\in\Pc$.
Of course, this notion extends verbatim to higher dimensional polyhedra.
Say that $\mu$ is \emph{$\Gs$-invariant} if $\mu(gP)=\mu(P)$ for all $g\in\Gs$.
For example, the Euclidean area is a valuation on bounded polygons
 that is invariant under the full group of Euclidean isometries.
Another important property of a valuation is being \emph{simple}:
 $\mu$ is said to be simple if it vanishes on degenerate polygons, i.e.,
 on line segments and on points.
Higher dimensional \emph{simple valuations} are defined as valuations vanishing on
 polyhedra of positive codimension.
If a polygon is cut into pieces, then the value of a simple valuation on the polygon
 can be computed as the sum of the corresponding values on all the pieces
 -- ignoring the fact that pieces may have nonempty intersections, since these intersections have positive codimension.
Next, when the pieces are rearranged using elements of $\Gs$, then $\Gs$-invariance
 guarantees that the value of $\mu$ does not change.
Thus, $\Gs$-invariant simple valuations always take the same value on all pairs
 of $\Gs$-scissors congruent polygons.

As follows from Zorn's lemma, there are many $\Q$-linear maps $\xi:\R\to\R$ that
 are not $\R$-linear; such maps are called \emph{additive functions} since
 they are characterized by the functional equation $f(x+y)=f(x)+f(y)$,
 which should hold for all $x$, $y\in\R$.
Every valuation $\mu$ with values in $\R$ gives rise to uncountably
 many valuations of the form $\xi\circ\mu$,
 where $\xi:\R\to\R$ is an additive function.
Instead of considering all $\xi\circ\mu$ along with $\mu$,
 one can confine oneself with valuations with values in $\Q$,
 having in mind that a valuation in any $\Q$-vector space reduces
 (by a choice of a basis) to a family of $\Q$-valued valuations.
Area can also be viewed in this way, as uncountably many $\Q$-valued valuations,
 since $\R$ is a vector space over $\Q$.

\textbf{Hilbert's third problem} asked whether an analog of the Wallace--Bolyai--Gerwien theorem holds in dimension 3,
 that is, whether any two convex 3-polyhedra of the same volume are scissors congruent
 (where the group of allowable transformations is the full isometry group of $\R^3$).
As was already conjectured by Hilbert, the answer is no, and the reason for this
 negative answer lies in the existence of many simple invariant valuations independent from the volume.
These were described by Dehn \cite{Deh01} and are known as the \emph{Dehn invariants}; see e.g. \cite{Sah79}.
Somewhat similar invariants can be found already in dimension 2, if one restricts $\Gs$ to
 the group $\Ts$ of all translations.
Namely, fix a cooriented line $L\subset \R^2$, where the \emph{coorientation} of $L$
 means a choice of a component of $\R^2\sm L$ that is called the \emph{positive side of} $L$
 (the other component is then called the \emph{negative side}).
Given a compact convex polygon $P$, define the number $\Upsilon_L(P)$ as the difference
 $\Len(P\cap L_+)-\Len(P\cap L_-)$, where $L_\pm$ are the two support lines of $P$
 parallel to $L$, and the sign is plus or minus depending on whether $P$ lies
 on the positive or on the negative side of the chosen support line.
By $\Len(I)$, we denote the Euclidean length of a compact interval $I$.
Valuations $\Upsilon_L$ in 2D were introduced by Hadwiger and Glur in their 1951 paper \cite{HG51}.

A principal result of \cite{HG51} is that two convex bounded polygons in the plane
 are $\Ts$-scissors congruent if and only if they have the same area and the
 same values of $\Upsilon_L$, for all $L$.
For example, a triangle and a square of the same area are not $\Ts$-scissors congruent.
On the other hand, all parallelepipeds of the same area are $\Ts$-scissors congruent
 since the valuations $\Upsilon_L$ vanish on them, for all $L$.
In particular, it is possible to ``rotate'' a square by an arbitrary angle by cutting
 it and translating the pieces (no rotations of pieces are allowed in the process).
Even though one needs infinite family of valuations $\Upsilon_L$ (parameterized by all lines)
 to distinguish every pair of polygons that are not $\Ts$-scissors congruent,
 one only needs finitely many valuations to decide whether two given polygons
 are $\Ts$-scissors congruent (it suffices to take the finitely many lines containing
 all the non-degenerate edges of both polygons).
By methods of abstract linear algebra, it can be deduced from \cite{HG51} that
 every simple $\Ts$-invariant valuation on planar polygons is a (possibly infinite)
 linear combination of $(\xi\circ\area)$ and $\xi\circ\Upsilon_L$,
 where $L$ ranges through all cooriented lines containing a given point (say, the origin),
 and $\xi:\R\to\Q$ ranges through all additive functions on $\R$ with values in $\Q$.

Results of \cite{HG51} were later generalized to real affine spaces of arbitrary dimension; see \cite{JT78,Sah79,KP16}.
Instead of $\Upsilon_L$, one considers the multidimensional \emph{Hadwiger invariants} $\Upsilon_\Lc$
 associated with oriented flags $\Lc$ in $\R^n$.
By a flag, we mean a finite collection of affine subspaces linearly ordered by inclusion.
Hadwiger invariants are based on flags $\Lc$ of a special types, namely, $\Lc$ should
 consist of subspaces $L^0\supset\dots \supset L^k$, where $L^i$ has codimension $i$ in $\R^n$,
 and $i$ takes all integer values between $0$ and $k$ inclusive.
Define the \emph{orientation} of $\Lc$ as the coorientation of $L^k$ in $\R^n$.
Oriented flags of this type form the set $\Lambda^\Oc_k$.
Given any $\Lc\in\Lambda^\Oc_k$ and any convex polytope $P\subset \R^n$ of dimension $n$,
 the value $\Upsilon_\Lc(P)$ is defined as $\pm\vol_{n-k}(P\cap L^k)$ if $P\cap L^i$ is a
 face of $P$ of codimension $i$, for all $i=0$, $\dots$, $k$, and as 0 otherwise.
Here, $\vol_{n-k}$ denotes the volume in dimension $n-k$ viewed as a
 measure on every affine subspace of codimension $k$, and the sign is plus
 if and only if the orientation of $\Lc$ defined by $P$ coincides with the given orientation of $\Lc$;
 otherwise, the sign is minus.
More precisely, assuming that $\Upsilon_\Lc(P)\ne 0$, consider vectors $v_{k-1}$, $\dots$, $v_{0}$,
 where each $v_i$ connects the barycenter of $P\cap L^{i+1}$ with the barycenter of $P\cap L^{i}$;
 these vectors form a basis in an affine subspace complementary to $L^k$ and thus define a coorientation of $L^k$;
 this is how $P$ defines an orientation of $\Lc$.
The space of all simple $\Ts$-invariant valuations on convex polytopes in $\R^n$ is
 explicitly described in \cite{KP16}: every such valuation is a (possibly infinite)
 linear combination of valuations of the form $\xi\circ\Upsilon_\Lc$, where $\Lc\in\Lambda^\Oc_k$,
 and $\xi:\R\to\Q$ is additive.
All linear relations between these valuations are also explicitly described.

$\Ts$-scissors congruence problem in $\R^n$, therefore, has a complete solution.
One can deduce from the above that two convex polytopes $P$ and $Q$ in $\R^n$
 are $\Ts$-scissors congruent if and only if all the Hadwiger invariants take the same value on $P$ and $Q$.
Note that the volume functional is among the Hadwiger invariants,
 namely, it corresponds to the flag $\{L^0=\R^n\}$.
For noncommutative groups $\Gs$, in particular, when $\Gs$ is the full group of
 Euclidean isometries, $\Gs$-scissors congruence problems remain open.
In particular, a full description of Euclidean scissors congruence invariants
 is still missing in dimensions $> 4$; the corresponding problems in hyperbolic
 and spherical spaces are open even in dimension 3.
Dehn invariants have been generalized to higher dimensional Euclidean spaces; however,
 in contrast to the theorem of Sydler that, in $\R^3$, two convex polytopes
 are scissors congruent if and only if they have the same volume and the
 same Dehn invariant, no extension of this fact is currently known for $\R^n$ with $n>4$
 (the case $n=4$ is covered by the result \cite{Jes68} of Jessen).
Further perspectives on scissors congruence problems are based on
 deep connections with algebraic K-theory, homological algebra, and algebraic geometry,
 see \cite{Zak16} for an overview.

\textbf{Rational convex polytopes} and valuations on them form an interesting subfield
 due to its connections with algebraic geometry and representation theory, specifically,
 with the theory of toric varieties.
Say that a polytope $P$ is \emph{rational} (resp., \emph{integer}) if all its
 vertices belong to $\Q^n\subset\R^n$ (resp., to $\Z^n\subset \R^n$).
In this context, a relevant group $\Gs$ is a lattice in the translation group $\Ts$, e.g., one can take
 all translations by vectors from $\Z^n\subset\R^n$; we write $\Gs=\Z^n$ in this case.
\emph{Valuations on integer polytopes} can be defined as the restrictions
 to the set of all integer polytopes of valuations defined on the set of all rational polytopes
 (P. McMullen \cite{McM09} proved that an a priori much weaker definition is equivalent to this;
 the difficulty is that the class of integer polytopes is not stable under taking intersections).
More generally, one can consider polytopes in ordered fields or even ordered rings;
 algebraic (in particular, homological) approach to scissors congruences extends easily to this setting,
 with understanding that polytopes are certain unions of simiplices.
Previous research on $\Z^n$-invariant valuations on all or rational polytopes includes \cite{McM78,PKh92a,McM93,Mor93};
 some of these works consider more general ground rings.
In \cite{Had13}, Hadwiger describes $\Z^n$-invariant valuations on all polytopes in $\R^n$.

\def\Scg{\mathrm{Sc}}

\textbf{The scissors congruence group} $\Scg(\Pc;\Gs)$ associated with a class $\Pc$
 of polytopes in $\R^n$ and a group $\Gs$ of isometries of $\R^n$ is an abelian group generated by
 members of $\Pc$ (given $P\in\Pc$, we write $[P]$ for the corresponding generator
 of $\Scg(\Pc;\Gs)$), subject to the following relations
\begin{itemize}
  \item 
  $[\0]=0$,
  \item 
  $[P\cup Q]=[P]+[Q]-[P\cap Q]$
 for $P$, $Q$, $P\cup Q\in\Pc$,
  \item 
  $[gP]=[P]$ for all $g\in\Gs$.
\end{itemize}
A further quotient of $\Scg(\Pc;\Gs)$ is obtained by imposing additional relations that
 $[P]=0$ whenever $\dim(P)<n$; this quotient is denoted by $\Scg^\circ(\Pc;\Gs)$ and is
 called a \emph{simple scissors congruence group}.
It is a classical result of Zylev \cite{Zyl65} that $P$, $Q\in\Pc$ are $\Gs$-scissors congruent if
 and only if $[P]=[Q]$ in $\Scg^\circ(\Pc;\Gs)$.
With this algebraic setup, one reduces a scissors congruence problem to
 a description of the relevant simple scissors congruence group.
Note that the space of all $\Q$-valued $\Gs$-invariant valuations on $\Pc$ is naturally isomorphic
 to the dual space $\Hom(\Scg(\Pc;\Gs);\Q)$, whereas the subspace of simple valuations
 is identified with $\Hom(\Scg^\circ(\Pc;\Gs);\Q)$.
For this reason, the problem of studying valuations and that of studying
  the corresponding scissors congruence group are essentially equivalent.
Most authors prefer the latter, since, in a practical sense,
 the scissors congruence groups are smaller than their duals.

Consider the map $P\mapsto [P]$ from $\Pc$ to $\Scg(\Pc;\Gs)$;
 this map can be viewed as a \emph{universal $\Gs$-invariant valuation}.
Universality here means that any $\Gs$-invariant valuation $\mu$ on $\Pc$
 with values, say, in some abelian group $A$, can be represented as $\mu(P)=\xi([P])$
 for a suitable group homomorphism $\xi:\Scg(\Pc;\Gs)\to A$.
P. McMullen showed in \cite{Mc89} that, if $\Pc$ is the set of all convex polytopes in $\R^n$
 and $\Ts$ is the group of all parallel translations, then, after a very minor modification,
 $\Scg(\Pc;\Ts)$ acquires a natural structure of a graded $\R$-algebra, the so called \emph{polytope algebra},
 in which the product is induced by Minkowski sum of polytopes.
The only modification needed is in degree 0.
Versions of McMullen's polytope algebra play key roles in intersection theory of toric varietes.
For example, as was shown by Fulton and Sturmfels in \cite{FS97}, \textbf{the ring of conditions
 of the complex torus} $(\C\sm\{0\})^n$ is naturally isomorphic to $\Scg(\Pc_\Z;\Z^n)$;
 the latter can be defined as the subgroup of $\Scg(\Pc_\Q;\Z^n)$ generated by integer polytopes.
Here, $\Pc_\Z$ is the set of all integer polytopes, and $\Pc_\Q$ is the set of all rational polytopes.
Our motivation for this paper was trying to better understand this isomorphism.

Algebraic approach to scissors congruence groups is based on the understanding that
 polytopes in $\R^n$ are unions of simplices.
This approach is nicely adapted to variations of the set where the vertices
 of polytopes are allowed to lie.
On the other hand, a dual approach is to fix some set $\Ac$ of affine hyperplanes in $\R^n$
 and to set $\Pc$ to consist only of those polyhedra that are bounded by hyperplanes from $\Ac$.
More precisely, for finite subsets $\Ac',\Ac''\subset \Ac$,
 consider the intersection of all hyperplanes from $\Ac'$ minus the union of all hyperplanes from $\Ac''$;
 the components of this set-theoretic difference are called \emph{$\Ac$-polyhedra}.
Bounded $\Ac$-polyhedra are called \emph{$\Ac$-polytopes}.
In this paper, we study scissors congruence problems for $\Ac$-polyhedra,
 and our approach is very different from the standard one, described above.
Write $\Pc_\Ac$ for the set of all $\Ac$-polytopes.

\textbf{Affine hyperplane arrangements} give an important special case of the setup just mentioned.
These are finite collections of affine hyperplanes in $\R^n$.
Given an affine arrangement $\Ac$, it makes sense to consider valuations on $\Pc_\Ac$,
 with $\Gs=\mathsf{0}$ being the trivial group, consisting only of the identity element.
In this case, the simple scissors congruence group identifies with the so called
 \emph{Varchenko--Gelfand algebra} \cite{VG87} of $\Ac$ over $\Q$ defined as the
 space of all $\Q$-valued functions on complementary components of $\Ac$.
Relations in the Varchenko--Gelfand algebra (more precisely, in the underlying rational vector space of it)
 are prototypical for relations in scissors congruence groups.
Even though elementary and almost tautological, this observation seems to have gone unnoticed.
Studying the group $\Scg^\circ(\Pc_\Ac;\mathsf{0})$ is possible by topological methods:
 it turns out that the relevant properties of this group are almost independent of
 the assumption that $\Ac$ consists of genuine hyperplanes.
For example, in dimension 2, the only properties of straight lines one needs
 are that every line is a simple curve properly embedded into $\R^2$,
 and that two distinct lines intersect transversely in a topological sense
 (one does not need any upper bound on the number of intersections).

The case $\Gs=\mathsf{0}$ is important since, for every other $\Gs$, the group $\Scg(\Pc;\Gs)$
 is a quotient of $\Scg(\Pc;\mathsf{0})$.
On the other hand, some other special choices of $\Gs$ admit nice topological interpretations.
For example, let $\Gs=\Z^n$ be the group of translations by integer vectors,
 and let $\hat\Ac$ be a set of rational (=defined over $\Q$) affine hyperplanes that is stable under the action of $\Z^n$.
Then $\hat\Ac$ defines a set $\Ac$ of codimension one subtori in the real torus $\R^n/\Z^n$;
 these subtori are geodesic hyperplanes with respect to the standard Euclidean structure on the torus.
If $\Ac$ is finite, then it is called a \emph{toric arrangement}.
Under the additional assumption that $\Ac$ contains at least one $n$-tuple of hyperplanes
 with a singleton intersection, one can consider polytopes in $\R^n/\Z^n$ bounded by
 hyperplanes of $\Ac$; let $\Pc_\Ac$ denote this class of polytopes.
Studying $\Z^n$-invariant valuations on $\Pc_{\hat\Ac}$ is now equivalent to
 studying valuations on $\Pc_\Ac$ in the torus.
Here, flat tori and their geodesic hyperplanes can be replaced by much
 more general topological manifolds and their topological hypersurfaces.
This is the approach we adopt in this paper.

\section{Introduction and statement of results}
One of our major aims is a functorial description of \emph{simple $\Z^n$-valuations on rational polytopes}.
Consider a set $\hat\Ac$ of rational hyperplanes in $\R^n$ that is \emph{$\Z^n$-invariant}, i.e.,
 is stable under parallel translations by the vectors of the integer lattice $\Z^n$
 and such that $\hat\Ac$ includes an $n$-tuple whose intersection is a single point.
Set $\Pc=\Pc_{\hat\Ac}$ to be the class of all $\hat\Ac$-polytopes.
A valuation $\mu$ on $\Pc$ is
 \emph{$\Z^n$-invariant} (we also call $\mu$ a \emph{$\Z^n$-valuation})
 if $\mu(P+v)=\mu(P)$ whenever $P\in\Pc$ and $v\in\Z^n$.
Recall that $\mu$ is \emph{simple} if $\mu(P)=0$ whenever $\dim(P)<n$.
One important special case is when $\hat\Ac$ consists of \emph{all} rational hyperplanes in $\R^n$.
For this paper, the other extremity, when $\hat\Ac$ is finite, will be no less ---
 methodologically, even more --- important.

\subsection{Examples}
Volume (normalized so that a fundamental domain 
 of $\Z^n$ has volume 1) gives
 a principal example $\vol:\Pc\to\Q$ of a $\Z^n$-valuation.
In fact, this valuation is simple and is invariant under all rational, not only integer, translations:
 that is, $\vol$ is a $\Q^n$-valuation.
The cardinality of $P\cap\Z^n$ is another example of a $\Z^n$-valuation; this one is not simple.
Fix a nonzero linear functional $\xi:\Q^n\to\Q$ and let $P_\xi$ for $P\in\Pc$
 be the face of $P$ where $\xi$ takes its maximum value.
Set $\mu_\xi(P):=\vol_{d-1}(P_\xi)$, where $\vol_{k}$ stands for the volume in dimension $k$;
 this is also a $\Q^n$-valuation, which is not simple.
Linear combinations of $\vol$ and $\mu_\xi$ for various $\xi$ are also $\Q^n$-valuations.
A related class of $\Q^n$-valuations can be obtained by
 fixing some $Q\in\Pc$ and setting $\mu_Q(P):=\vol(P+Q)$, where $+$ in the right-hand side is the Minkowskii sum
 (the fact that $\mu_Q$ are indeed valuations is true and simple but
  not completely obvious, cf. \cite{MS83,Had13}).

\subsection{Variations of the setup}
As was explained in Section \ref{s:over}, varying $\Pc$ and $\Gs$, one obtains several interrelated contexts.
The following is a thematic index of previous works on similar topics:
\begin{description}
  \item[\cite{HG51,JT78,Sah79,KP16}] all polytopes, simple $\R^n$-valuations;
  \item[\cite{Mc89,PKh92a,McM93,Mor93,Dup01}] all polytopes, $\R^n$-valuations;
  \item[\cite{Dup82,PKh92a,Dup01}] all polytopes, all (or all simple) valuations;
  \item[\cite{McM78,PKh92a,McM93,Mor93}] integer or rational polytopes, $\Z^n$-valuations;
  \item[\cite{Had13}] all polytopes, $\Z^n$-valuations.
\end{description}
Papers \cite{JT78,Sah79,Mc89,Mor93} also contain generalizations of the indicated setups to
 ordered fields or even more general ground rings instead of $\R$ and $\Z$.

\def\Val{\mathrm{Val}}

\subsection{Generators and relations}
\label{ss:gens-rels}
Go back to the initial setup: $\hat\Ac$ is a $\Z^n$-invariant collection of rational hyperplanes in $\R^n$
 that includes an $n$-tuple with a singleton intersection, and $\Pc=\Pc_{\hat\Ac}$
 is the class of all $\hat\Ac$-polytopes.
The rational vector space $\Val=\Val(\hat\Ac;\Z^n)$ of all simple $\Z^n$-va\-lu\-ations on $\Pc$
 can be described in terms of generators and relations.
We allow for infinite linear combinations of generators: a sum $\sum_\la f(\la)\mu_\la$
 over a family of valuations $\mu_\la\in\Val$ makes sense if, for every $P\in\Pc$,
 there are only finitely many values of $\la$ with $\mu_\la(P)\ne 0$.
Say that the family $(\mu_\la)_\la$ \emph{generates} $\Val$ if every $\mu\in\Val$
 can be written as $\sum_\la f(\la)\mu_\la$ with some rational coefficients $f(\la)\in\Q$.
Also, a \emph{relation} between the generators $\mu_\la$ is an equality of the form $\sum_\la f(\la)\mu_\la=0$,
 where the sum may be infinite but, again, for every $P\in\Pc$, only finitely many terms take nonzero values on $P$.

\subsection{Functoriality}
\label{ss:functor}
We want to describe a system of generators and defining relations for $\Val$ that
 would depend \emph{functorially} on $\hat\Ac$.
Namely, for every $\hat\Ac$, we define a distinguished generating set $\Gc(\hat\Ac)$ of $\Val(\hat\Ac;\Z^n)$.
Every map $j$ from a set $\hat\Ac$ of hyperplanes to another such set $\hat\Ac'$
 such that $H$ and $j(H)$ are parallel, for all $H\in\hat\Ac$, while $j(H')$ is a $\Z^n$-translate of $j(H)$
 whenever $H'$ is a $\Z^n$-translate of $H$,
 induces a map $j_*:\Gc(\hat\Ac)\to \Gc(\hat\Ac')$ at the level of the generators.
Composition of maps should correspond to composition of induced maps, that is,
 $(j_1\circ j_2)_*=j_{1*}\circ j_{2*}$
 for any pair of maps $j_1:\Gc(\hat\Ac_1)\to \Gc(\hat\Ac_2)$ and $j_2:\Gc(\hat\Ac_2)\to\Gc(\hat\Ac_3)$.
If $\hat\Ac$ is in general position, it can be shown that $j_*$ extends to an additive
 group homomorphism from $\Val(\hat\Ac;\Z^n)$ to $\Val(\hat\Ac';\Z^n)$.
Functoriality is need, for example, when one wants to use the ``same'' generating
 set for different subgroups of translations.
In a sequel of this paper, the authors plan to study \emph{polynomial valuations},
 i.e., valuations $\mu$ such that, for every $P\in\Pc$, the function $v\mapsto \mu(P+v)$
 is a polynomial function of $v\in\R^n$ or $v\in\Z^n$.
For a description of polynomial valuations, a functorial choice of generators plays a decisive role.
As it turns out, a version of the classical system of Hadwiger invariants \cite{HG51} has the
 desired functorial property, see \ref{ss:flag-val}.
Functoriality of generators was first addressed by Varchenko \cite{Var87} in
 the context of finite affine hyperplane arrangements.

\subsection{A general paradigm}
\label{ss:parad}
A principal novelty of this paper lies not so much in results
(which generalize known results) but in the approach, allowing to uncover
 connections between different disciplines.
Our line of argument can be roughly expressed as follows:
$$
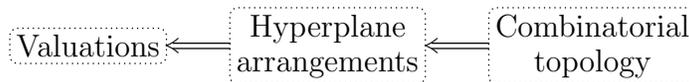
\xymatrix{
*+<5pt>[F.:<3pt>]{\txt{Valuations}}
& *+<5pt>[F.:<3pt>]{\txt{Hyperplane\\ arrangements}} \ar @{ =>} [l]
& *+<5pt>[F.:<3pt>]{\txt{Combinatorial\\ topology}}\ar @{ =>} [l]
}
$$
The main theorem stated below should rather be thought of as
 the main example of the techniques.
Variations of the main theorem with topological flavor are possible.
In case a collection $\Ac$ of affine hyperplanes in $\R^n$ is finite,
 the description of valuations on $\Ac$-polytopes is closely related to
 the description of the Varchenko--Gelfand algebra \cite{VG87} for $\Ac$.
When a $\Z^n$-invariant collection $\hat\Ac$ of rational hyperplanes is locally finite,
 we obtain a finite collection of codimension one subtori in the compact $n$-torus.
Studying $\Z^n$-valuations on $\hat\Ac$-polytopes is again related to
 studying a toric analog of the Varchenko--Gelfand algebra.
Moreover, these questions make sense in much more general topological situation,
 where hyperplanes and flats do not have to be straight and may even have nontrivial topology.

\subsection{Flags in the torus}
Any hyperplane from $\hat\Ac$ defines a codimension one subtorus in $\T^n:=\R^n/\Z^n$.
Denote the set of all such subtori by $\Ac$ and call the elements of $\Ac$ \emph{hyperplanes in $\T^n$}.
By a \emph{flat} of $\Ac$ (or an $\Ac$-flat),
 we mean a component of the intersection $\bigcap\Ac'$, where $\Ac'\subset\Ac$
 is a subfamily of $\Ac$.
Note that an intersection of hyperplanes may be disconnected;
 however, all flats are connected by definition.
In particular, all of $\T^n$ is a flat, as the intersection of the empty subfamily.
Also, each hyperplane of $\Ac$ is a flat; by our assumption on $\Ac$,
 there are flats of all dimensions from $0$ to $n$.
Define a \emph{flag} (of $\Ac$) as a set of flats that is linearly ordered by inclusion.
We will mostly consider flags of the form
$$
\Lc=(L^0\supset\dots\supset L^k),
$$
where $L^i$ is a flat of codimension $i$, and $i$ ranges through all integers from $0$ to $k$.
Write $[\Lc]$ for the smallest flat $L_k$ of the flag.
Let $\Lambda_k$ be the set of all flags $\Lc$ as above.

\subsection{Hadwiger invariants}
\label{ss:flag-val}
Consider an $\hat\Ac$-polytope $P$ and a chain of faces $F^0\supset\dots\supset F^k$,
 where $F^i$ is a codimension $i$ face of $P$, and $i$ ranges from 0 to $k$.
Letting $L^i$ be the affine span of $F^i$ projected to the torus $\T^n$,
 we obtain a flag $\Lc=(L^0\supset\dots\supset L^k)$ from $\Lambda_k$
 called a \emph{flag of $P$}.
Also, $P$ defines a coorientation of $L^k$ as follows.
Take a vector $v_i$ from the barycenter of $F^k$ to the barycenter of $F^i$, for all $i=0$, $\dots$, $k-1$.
In this way, one obtains a basis in some transversal of $L^k$, which defines a coorientation of $L^k$.
By an \emph{oriented flag}, we mean a flag $\Lc$ with a chosen
 orientation of $[\Lc]$.

Given an oriented flag $\Lc\in\Lambda_k$, 
 define $\Upsilon_\Lc(P)$ to be $\pm \vol_{n-k}(P\cap [\Lc])$ if $\Lc$ is a flag of $P$ obtained as above,
 the sign being plus or minus depending on whether the chosen coorientation of $[\Lc]$
 coincides with that induced from $P$.
If $\Lc$ is not a flag of $P$, one sets $\Upsilon_\Lc(P)=0$.
Note that $\Upsilon_\Lc$ is a simple $\Z^n$-invariant valuation on $\hat\Ac$-polytopes.
Call $\Upsilon_\Lc$ the \emph{Hadwiger invariant of rank $k$} associated with a flag $\Lc$ and
 a coorientation of $[\Lc]$.

\begin{thm*}
The Hadwiger invariants $\Upsilon_\Lc$ generate the space $\Val(\Pc_{\hat\Ac},\Z^n)$: every $\Z^n$-invariant simple valuation
 $\mu$ on $\hat\Ac$-polytopes can be written as $\mu_f=\sum_{\Lc} f(\Lc) \Upsilon_\Lc$
 with some rational coefficients $f(\Lc)\in\Q$.
\end{thm*}

Related results are \cite{McM78,JT78,Sah79,Mc89,Mor93}; 
 the common special case of the above theorem
 and \cite{McM78,Mor93} is when $\hat\Ac$ is the collection of all rational hyperplanes.
Observe that $f(\Lc)$ depends not only on the geometric flag $\Lc$ but also
 on its orientation; one can make $f$ \emph{alternating},
 i.e., changing sign as the orienation of $\Lc$ reverses.
It now remains to describe relations between $\Upsilon_\Lc$, equivalently,
 to characterize those alternating functions $f$, for which $\mu_f=0$.

Note that Hadwiger invariants depend functorially on $\hat\Ac$ in the sense of \ref{ss:functor}.
Indeed, a map $j:\hat\Ac\to\hat\Ac'$ such that $j(H)$ is parallel to $H$, for every $H\in\hat\Ac$,
 while $j(H')$ is a $\Z^n$-translate of $j(H)$ whenever $H'$ is a $\Z^n$-translate of $H$,
 gives rise to a map between the corresponding sets of subtori, i.e., a map from $\Ac$ to $\Ac'$.
Next, one extends this map to the flats so that intersections correspond to intersections,
 and to (oriented) flags.
Since Hadwiger invariants are labeled by oriented flags,
 we also have a natural map between Hadwiger invariants.

\subsection{Reciprocity laws}
Suppose that a flag $\Lc^\circ$ contains flats of all codimensions from 0 to $k$
 except exactly one codimension $m$, strictly between 0 and $k$.
There are then several (possibly, infinitely many) ways of choosing a codimension $m$
 flat $M$ to complete $\Lc^\circ$ so that $\Lc^\circ\cup\{M\}\in\Lambda_k$.
Fix some coorientation of $[\Lc^\circ]$.
It is not hard to verify the following relation between Hadwiger invariants:
$$
\sum_{M} \Upsilon_{\Lc^\circ\cup \{M\}}=0,
$$
 where the sum is over all codimension $m$ flats $M$ such that $\Lc^\circ\cup\{M\}\in\Lambda_k$,
 and all flags $\Lc^\circ\cup\{M\}$ from the sum come with the same coorientation of $[\Lc^\circ]$ chosen above.
Call this relation the \emph{reciprocity law} at $\Lc^\circ$.
Each $\Lc^\circ$ gives rise to its own reciprocity law.
When the coorientation of $[\Lc^\circ]$ gets reversed, the left-hand side of the equality changes sign,
 so that one obtains an equivalent relation.
Reciprocity laws as stated above mimic Parshin's reciprocity laws \cite{Par76,Lom82,BM96,Kho08} and, indeed, are
 closely related to the latter in the case of affine hyperplane arrangements or
 toric arrangements.

\subsection{Period vanishing conditions}
\label{ss:0pervan}
So far, it was not important how the $k$-volumes in $k$-dimensional flats of $\Ac$ were normalized.
Now, we fix a normalization.
Given a flat $L$ of $\Ac$, normalize the volume form on $L$ so that $L$ has total volume one.
In the discussion that follows, all codimensions are relative to $\T^n$.

Fix a flag $\Lc\in\Lambda_{m-1}$ and an oriented simple closed curve $\ga\subset [\Lc]$.
Another relation between the Hadwiger invariants has the form
$$
\sum_{M\subset [\Lc]} (\ga\cdot M)\Upsilon_{\Lc\cup\{M\}}=0.
$$
Here, the (possibly infinite) summation is over all codimension $m$ flats $M\subset [\Lc]$ such that
 $\Lc\cup\{M\}\in\Lambda_m$, and $\ga\cdot M$ denotes the intersection index of $\ga$ and $M$.
In order to make the summands well defined, one needs to fix coorientations of all $M$s relative to $\T^n$:
 to this end, fix some coorientation of $[\Lc]$, and, for each $M$, choose
 some coorientation of $M$ relative to $[\Lc]$.
Note: reversing the relative coorientation of $M$ means changing signs of
 both $\Upsilon_{\Lc\cup\{M\}}$ and $\ga\cdot M$, hence it does not influence the sum.
Relations just displayed are called \emph{period vanishing conditions}.
These relations are closely related to Cauchy's theorem in convex geometry
 (linear relations on the volumes of facets of a convex polytope).
Since the intersection index $\ga\cdot M$ depends only on the homology class of $\ga$,
 the number of independent period vanishing conditions for a given $\Lc$
 equals the first Betti number of $[\Lc]$.

\subsection{Main theorem}
\label{ss:main-thm}
The following result includes the theorem stated in \ref{ss:flag-val} but gives more
 (it describes all relations between Hadwiger invariants).

\begin{thm*}
The space $\Val(\Pc_{\hat\Ac};\Z^n)$ of all $\Z^n$-invariant simple valuations on $\hat\Ac$-polytopes is generated by Hadwiger invariants.
Relations between Hadwiger invariants are generated by the reciprocity laws and the period vanishing conditions.
\end{thm*}

Here, we allow for infinite linear combinations of generators as well as infinite linear combinations of relations.
Note that the generating set of $\Val(\Pc_{\hat\Ac};\Z^n)$ is stable with respect to \emph{all} rational translations.
A special case of the main theorem where $\hat\Ac$ consists of all rational hyperplanes in $\R^n$ is close to \cite{Mor93}.
However, firstly, we use completely different methods, of more topological flavor.
Secondly, our description concerns all simple $\Z^n$-invariant valuations on rational polytopes
 rather than only their restrictions to the class of integer polytopes.
In \cite{McM93}, there is another closely related result, where all, not necessarily simple, valuations are considered.

Versions of the main theorem are possible, and provable in the same way,
 where rational hyperplanes from $\Ac$ are replaced with hypersurfaces in $\T^n$ satisfying
 certain natural topological assumptions.
For example, if $n=2$, it suffices to consider any collection of essential simple closed
 curves in $\T^2$ such that, firstly, any pair of curves from the given collection
 intersects transversely (in a topological sense) and, secondly, the homology classes of
 these curves generate $\Hs_1(\T^2)$.

\subsection{All simple valuations}
As already explained, the choice of the ``main theorem'' is somewhat arbitrary since
 the main point is a new method. 
We give other illustrations of the method.
For example, all simple valuations on $\Ac$-polytopes and all the relations between them
 are described in \ref{ss:val-polyt} for arbitrary collections $\Ac$ of affine hyperplanes in $\R^n$.
Namely, let now $\Lc$ be a complete flag (including flats of $\Ac$ of all dimensions between 0 and $n$)
 in $\R^n$ (rather than in the torus), equipped with some coorientation of $[\Lc]$,
 and define $\Upsilon_\Lc$ similarly to the above: $\Upsilon_\Lc(P)=\pm 1$ if $\Lc$ is a flag of $P$ and 0 otherwise.
In the former case, the sign depends on whether the coorientation of $[\Lc]$ induced
 by $P$ is the same as the chosen one.
Hadwiger invariants (now defined only for complete flags) generate the space $\Val(\Pc_\Ac;\mathsf{0})$
 of all simple valuations on $\Ac$-polytopes.
Relations between Hadwiger invariants are generated by the reciprocity laws and
 the period vanishing conditions.
See \ref{ss:val-polyt} for further details; note also that the special case where $\Ac$
 contains all affine hyperplanes in $\R^n$ is equivalent to \cite[Theorem 3.5]{Dup01}.
Period vanishing conditions for complete flags look especially simple.
Given an almost complete flag $\Lc$ that misses only the dimension 0 term,
 the corresponding relation is $\sum_{a\in [\Lc]} \Upsilon_{\Lc\cup\{a\}}=0$,
 where the sum is over all dimension 0 flats $\{a\}\subset[\Lc]$, and
 coorientations of $a$ are chosen consistently.

In \ref{ss:val-polyh}, we consider all cosimple simple valuations on $\Ac$-polyhedra, not necessarily bounded.
Valuations being \emph{cosimple} means that they vanish on polyhedra containing
 affine subspaces of positive dimension.
This setting is perhaps farther from classical theory of scissors congruences, since the
 latter views polytopes as figures which can be dissected into simplices.
On the other hand, for finite $\Ac$, the result of \ref{ss:val-polyh} can be deduced from \cite[Theorem 2.4]{SV91} and
 a graded vector space isomorphism \cite{VG87} between the Orlik--Solomon algebra and the Varchenko--Gelfand algebra.


\subsection{Topological arrangements}
All announced results remain true if a collection $\Ac$ of hyperplanes in an affine space or
 in a torus is replaced with a collection of topological hypersurfaces, subject to minor topological requirements.
We now restrict our attention to the case where $\Ac$ is finite.
One needs to assume that every flat (a component of an intersection of hyperplanes)
 is also a toplogical manifold and that hypersurfaces from $\Ac$ cut each flat into topological cells.
In this case, $\Ac$ is called a \emph{topological arrangement} (see \ref{ss:top-arr} for a formal definition).
The ambient space does not have to be $\R^n$ or $\T^n$ either; it can be an arbitrary topological manifold.
An important virtue of this topological approach is that reciprocity laws and
 period vanishing conditions acquire natural cohomological interpretations.

Topological arrangements 
 generalize both affine hyperplane arrangements
 and toric arrangements \cite{dCP05,CDD+20}.
Specialized to the case of toric arrangements, our methods yield a description
 for the toric analog of the Varchenko--Gelfand algebra \cite{VG87} (see Section \ref{s:VGalg}).
Even in the case when all flats have trivial topology, the generalization provided
 by topological arrangements is meaningful: it includes, e.g., arrangements of ``pseudolines''.
For example, in the classical Pappus arrangement, one can deform one
 line locally to avoid one of the vertices in the class of pseudoline arrangements,
 even though the arrangement obtained is not realizable over any field.
More than that, in the plane, one can take any collection of simple curves, ``lines'',
 going from infinity to infinity and intersecting transversely in a topological sense
 at finitely many points; two different ``lines'' may even have several intersection points
 (see Section \ref{ss:draw}).

\subsection{Organization of the paper}
Section \ref{s:cc} restates the main theorem and some related results in a pre-dual language of convex chains.
The main theorem follows from claims made in Section \ref{s:cc} by methods of linear algebra discussed in Section \ref{s:linalg}.
In Section \ref{s:semireg}, we lay a 
 foundation for studying topological
 arrangements; the latter is initiated in Section \ref{s:top-arr}.
By definition, topological arrangements form a special class of CW spaces.
Any CW space gives rise to the corresponding algebraic (co)chain complex whose
 (co)homology is the (co)homology of $X$.
For topological arrangements, one obtains several algebraic (co)chain complexes.
A relationship between those is important for our main results.

Given a hyperplane arrangement $\Ac$ in $\R^d$, \emph{the Varchenko--Gelfand algebra} of $\Ac$
 is by definition the algebra of locally constant functions on the complement of $\bigcup\Ac$.
Section \ref{s:VGalg} discusses generalizations of the Varchenko--Gelfand algebras to topological arrangements.
While the definition of the algebra itself extends in a straightforward way,
 the degree filtration requires an interpretation.
Quotients of the degree filtration are studied in Section \ref{s:rels} via
 their embeddings into spaces of alternating functions on oriented flags.
Finally, Sections \ref{s:aff} and \ref{s:tor} deduce properties of
 the spaces of convex chains for arbitrary collections of affine or rational toric hyperplanes
 from those of finite arrangements.
In particular, the proof of the main theorem concludes in Section \ref{s:tor}.

\section{Convex chains}
\label{s:cc}
Below, we take a viewpoint on valuations as linear functionals on relevant spaces of convex chains.
Description of these spaces yields description of valuations by dualization.

\subsection{Dimension one}
Before discussing general methods, we describe the simplest special case, namely, that of dimension 1.
Let $\hat\Ac$ be any collection of points in $\R^1$; then $\hat\Ac$-polytopes are
 compact intervals with endpoints in $\hat\Ac$.
Nontrivial flags correspond to the case $k=1$ and identify with single points of $\hat\Ac$.
Oriented flags are points of $\hat\Ac$ equipped with directions of motion through these points;
 a standard orientation is from left to right.
Given such an oriented flag $\Lc=\{a\}\subset\hat\Ac$, the Harwiger invariant $\Upsilon_\Lc$ takes value
 $1$ on any interval $[b;a]$ with $b<a$, value $-1$ on any interval $[a;b]$ with $a<b$, and value 0
 on all other intervals.
An infinite linear combination $\sum_{a\in\hat\Ac} f(a)\Upsilon_{\{a\}}$ corresponds to
 a function $f:\hat\Ac\to\Q$; in other words, every such function $f$ defines a valuation $\mu_f$.
More explicitly, $\mu_f[a;b]=f(b)-f(a)$ whenever $a<b$ and $a$, $b\in\hat\Ac$.
Clearly, any simple valuation on $\hat\Ac$-polytopes can be represented in this form.
The only relation is the period vanishing condition $\mu_1=0$, where $1$ in the subscript
 denotes the constant function with value 1.
In other words, the space of all simple valuations on $\hat\Ac$-polytopes can be described
 as the space of functions on $\hat\Ac$, modulo constants.

Suppose now that $\hat\Ac$ is a $\Q$-vector subspace of $\R$, and a simple valuation $\mu_f$
 is $\hat\Ac$-\emph{invariant}, i.e.,  invariant under all translations by vectors from $\hat\Ac$.
Under these assumptions, it is enough to choose $f$ so that $f(0)=0$ (which is
 always possible since adding a constant to $f$ does not alter $\mu_f$);
 this choice will guarantee that $f$ is linear over $\Q$ --- an easy consequence of
 the additivity of $\mu_f$ together with its translation invariance.
Thus, the $\Q$-vector space of all $\hat\Ac$-invariant simple valuations on $\hat\Ac$-polytopes
 can be identified with the space of all $\Q$-linear functions on $\hat\Ac$.

However, the context of our main theorem is when $\hat\Ac$ is stable under $\Z$-translations,
 and the valuations considered are assumed to be simple and $\Z$-invariant.
Any such valuation can be written as $\mu_f=\sum_{a\in\Ac} f(a)\Upsilon_{\{a\}}$, where
 $f$ is now a $\Q$-valued function on $\Ac$, the projection of $\hat\Ac$ onto $\R/\Z$.
Equivalently, $f$ can be thought of as a $\Z$-periodic function on $\hat\Ac$.
Again, the only relation is that $\mu_1=0$.
The space $\Val(\Pc_{\hat\Ac};\Z)$ identifies with the space of all functions on $\Ac$,
 modulo constants.

\subsection{Space of convex chains}
Fix a collection $\Ac$ of rational hyperplanes (=codimension one subtori) in $\T^n$ and assume that
 $\Ac$ includes an $n$-tuple of hyperplanes with intersection of dimension 0.
By \emph{$\Ac$-polytopes}, we mean topological cells obtained as components of
 $H_1\cap\dots\cap H_k\sm (H_{k+1}\cup\dots\cup H_m)$, for some $H_1$, $\dots$, $H_m\in\Ac$
 and some $k<m$.
Let $\Vc=\Vc(\Ac)$ be the subspace of $L^\infty(\T^n;\Q)$ represented by functions
 $f:\T^n\to\Q$ with $f(\T^n)$ finite and such that $f^{-1}\{c\}$ is a finite union of $\Ac$-polytopes,
 for every $c\in\Q$.
As is indicated by the membership $f\in L^\infty$, two functions as above are identified
 if they differ on a measure zero set, say, on a union of finitely many hyperplanes from $\Ac$.
Call $\Vc$ the \emph{space of convex chains} associated with $\Ac$.
This space is spanned by the indicator functions of open $\Ac$-polytopes
 (recall that, for a subset $A\subset\R^n$, the \emph{indicator function} $\1_A$ of $A$
 is the function on $\R^n$ that takes value 1 at all points of $A$ and value 0 at all other points).
Simple $\Z^n$-invariant valuations can be viewed as linear functionals on $\Vc$, see \cite{Vol57,PS70,Gro78};
 thus $\Val(\Pc_{\hat\Ac};\Z^n)=\Vc^*$ is the dual space of $\Vc$ over $\Q$.
Often, it is preferable to study $\Vc$ instead of $\Vc^*$, as this space is somewhat smaller.
If $\al\in\Vc$ and $\mu\in\Vc^*$, then $\<\mu,\al\>$ stands for the
 value of the natural bilinear nondegenerate pairing between $\Vc^*$ and $\Vc$.
Other versions of convex chains are discussed in \cite{Gro78,PKh92a}.

\subsection{Elementary chains}
\label{ss:elem}
Given a flat $L$ of $\Ac$, write $\Ac|_L$ for the collection of hyperplanes in $L$
 formed by components of intersections $L\cap H$, where $H\in\Ac$ are such that $L\not\subset H$.
Consider a flag $\Lc\in\Lambda_k$, i.e., a flag of $\Ac$ of the form $L^0\supset\dots\supset L^k$,
 where $L^i$ has codimension $i$, and $i$ ranges through all integers from $0$ to $k$.
Let $E^k_\Lc$ stand for the space of convex chains in $[\Lc]$ associated with $\Ac|_{[\Lc]}$,
 twisted by the coorientations of $[\Lc]$.
\emph{Twisting} means the following: if a coorientation of $[\Lc]$ is fixed, then $E^k_\Lc$ identifies
 with the space of convex chains in $[\Lc]$.
On the other hand, reversing the coorientation of $[\Lc]$ means that the corresponding identification changes sign.
Define $E^k$ as the direct sum of $E^k_\Lc$ over all $\Lc\in\Lambda_k$.
Recall that every element $\xi\in E^k$ has the form
$$
\xi=\bigoplus_{\Lc\in\Lambda_k} \xi_\Lc,
$$
 where only finitely many terms $\xi_\Lc$ are nonzero;
 the term $\xi_\Lc$ is called the \emph{$\Lc$-component} of $\xi$.
Spaces $E^k$ are called \emph{spaces of elementary chains}.

\subsection{Leray operators}
\label{ss:Ler-op}
We now define certain linear maps $D:E^k\to E^{k+1}$.
For $\xi\in E^k$ and a flag $\Lc\in\Lambda_{k+1}$, the $\Lc$-component
 $(D\xi)_\Lc\in E^{k+1}_\Lc$ of $D\xi$ is defined as follows.
Recall that $[\Lc]$ is the smallest by inclusion flat of $\Lc$, and let $\Gamma\Lc\in\Lambda_k$
 be the flag obtained from $\Lc$ by dropping the term $[\Lc]$.
Having chosen some coorientaion of $[\Gamma\Lc]$ and some coorientation of $[\Lc]$ relative to $[\Gamma\Lc]$,
 we can identify $\xi_{\Gamma\Lc}$ with a function on $[\Gamma\Lc]$, and talk about the positive and
 the negative sides of $[\Lc]$ in $[\Gamma\Lc]$.
Also, our choices give rise to some coorientation of $[\Lc]$ in $X$,
 the \emph{composition} of the chosen coorientation of $[\Gamma\Lc]$ and
 the chosen relative coorientation of $[\Lc]$ in $[\Gamma\Lc]$.
Take the limit of $\xi(q)$ as $q\in [\Gamma\Lc]$ approaches the flat $[\Lc]$ from its positive side
 and subtract the limit of $\xi(q)$ as $q$ approaches the flat $[\Lc]$ from the negative side.
Set the result of this discrete differentiation to be $(D\xi)_\Lc$,
 identified with an element of $E^{k+1}_{\Lc}$ using the above mentioned coorientation of $[\Lc]$.
It is a function on $[\Lc]$, well defined outside a union of codimension $>k+1$ flats.
Reversing the coorientation of $[\Gamma\Lc]$ means replacing $\xi$ with $-\xi$,
 hence also replacing $(D\xi)_\Lc$ with $-(D\xi)_\Lc$.
Changing the relative coorientation of $[\Lc]$ in $[\Gamma\Lc]$ swaps the sides of $[\Lc]$,
 hence the difference changes sign; again, $(D\xi)_\Lc$ is replaced with $-(D\xi)_\Lc$.
Thus, indeed, $(D\xi)_\Lc$ is a well-defined element of $E^{k+1}_\Lc$.
Call $D$ the \emph{Leray operator}; it is a version of discrete differential.
We may write $D_k$ instead of $D$ if the dependence of $D$ on $k$ needs to be emphasized.

\subsection{Spaces $\Fc_k$ and the degree filtration}
\label{ss:degfilt0}
Let $\Fc_k$ denote the kernel of $D:E^k\to E^{k+1}$.
If $\xi\in\Fc_k$, then, for every flag $\Lc\in\Lambda_k$, the $\Lc$-component $\xi_\Lc$ of $\xi$
 identifies with a constant function on $[\Lc]$.
For this reason, $\xi$ can be viewed as an \emph{alternating function on the set $\Lambda^\Oc_k$
 of oriented flags from $\Lambda_k$}, where an orientation of $\Lc\in\Lambda_k$ means
 a choice of a coorientation of $[\Lc]$, and $\xi$ being alternating means that it changes sign
 whenever $\Lc$ reverses its orientation.
Define $\Vc_{\le k}$ as the kernel of $D^{k+1}:\Vc=E^0\to E^{k+1}$.
These spaces form a filtration
$$
\Vc_{\le 0}\subset \Vc_{\le 1}\subset \dots \subset \Vc_{\le n}=\Vc
$$
 of $\Vc$ called the \emph{degree filtration}.
Setting $\Vc_{<k}:=\Vc_{\le k-1}$, consider the \emph{quotients} $\Vc_k:=\Vc_{\le k}/\Vc_{<k}$ of the degree filtration.
By definition of the degree filtration, $D^k$ embeds the quotient $\Vc_k$ into $\Fc_k$.
We identify $\Vc_k$ with its image in $\Fc_k$.
Our next objective is a description of defining linear equations for $\Vc_k$ as a subspace of $\Fc_k$.

\subsection{Reciprocity laws in $E^k$}
\label{ss:recipEk}
Consider a flag $\Lc^\circ$ of $\Ac$ that includes flats of all codimensions from 0 to $k$
 except just one 
 codimension $m$ with $0<m<k$.
An element $\xi\in \Fc_k$ satisfies the \emph{reciprocity law} at $\Lc^\circ$ if
 the sum of $\xi_{\Lc^\circ\cup\{M\}}$
 over all codimension $m$ flats $M$ such that $\Lc^\circ\cup\{M\}\in\Lambda_k$, vanishes.
Here, the orientations of all flags $\Lc^\circ\cup\{M\}$ are given by the same coorientation of $[\Lc^\circ]$,
 and, using these orientations, all terms $\xi_{\Lc^\circ\cup\{M\}}$ identify with rational numbers.
Note: even though the sum may be infinite, it has only finitely many nonzero terms.
It is not hard to see that all elements of $\Vc_k$ satisfy all the reciprocity laws,
 i.e., reciprocity laws at all flags $\Lc^\circ$ as described above.

More generally, one can state reciprocity laws for $\xi\in E^k$;
 the sum of $\xi_{\Lc^\circ\cup\{M\}}$ is no more a rational number but
 a convex chain on $[\Lc]$; the reciprocity law at $\Lc^\circ$ stipulates
 that this convex chain vanishes.

\subsection{Period vanishing in $E^k$}
\label{ss:pervanEk}
Let $\xi=\bigoplus_\Lc \xi_\Lc$ be an element of $E^k$, for $k>0$.
Now fix $\Lc^\circ\in\Lambda_{k-1}$ and consider all flags $\Lc\in\Lambda_k$ that include $\Lc^\circ$.
Set $\xi_{\circ}$ to be the direct sum of $\xi_\Lc$ over all $\Lc\in\Lambda_k$ such that $\Lc\supset\Lc^\circ$.
Even though $\xi_{\circ}$ is not a smooth differential 1-form on $[\Lc^\circ]$,
 one can define the \emph{integral} of $\xi_{\circ}$ over a smooth simple path $\ga$ in $[\Lc^\circ]$
 in general position with respect to all $[\Lc]$ such that $\Lc^\circ\subset\Lc$: it is equal
 to the sum of $\xi_\Lc(\ga(t))$ for all $(t,\Lc)$ such that $\ga(t)\in [\Lc]$.
Here, a pair $(t,\Lc)$ with $\ga(t)\in [\Lc]$ defines a relative coorientation of $[\Lc]$ in $[\Lc^\circ]$ such that $\ga$ crosses
 $[\Lc]$ from the negative side to the positive side at time $t$;
 together with some fixed coorientation of $[\Lc^\circ]$ in $\T^n$,
 these relative coorientations of $[\Lc]$ define orientations of $\Lc$ used in the computation of $\xi_\Lc(\ga(t))$.
Assume that $\xi$ satisfies the reciprocity laws; it is then straightforward to verify that
 the integral of $\xi_{\circ}$ over any boundary vanishes, i.e., $\xi_{\circ}$ is a cocycle.
It follows that $\xi_{\circ}$ represents a first cohomology class $[\xi_{\circ}]$ of $[\Lc^\circ]$.
By definition, $\xi$ \emph{satisfies the period vanishing condition at} $\Lc^\circ$ if $[\xi_{\circ}]=0$,
 that is, if $\xi_{\circ}$ is a coboundary.
Saying that $\xi$ satisfies \emph{all} the period vanishing conditions (i.e., period vanishing conditions at all $\Lc^\circ\in\Lambda_{k-1}$)
 is equivalent to saying that $\xi=D\eta$ for some $\eta\in E^{k-1}$; see Section \ref{ss:pervan}.
Our description of the period vanishing conditions in $E^k$ given here is consistent
 with that given in \ref{ss:0pervan} for $\Fc_k$.

\subsection{A description of $\Vc_k$}
\label{ss:descVck}
The following is a pre-dual version of the main theorem, cf. \ref{ss:main-thm}.

\begin{thm*}
The subspace $\Vc_k\subset\Fc_k$ consists of all elements of $\Fc_k$
 that satisfy all the reciprocity laws and all the period vanishing conditions.
\end{thm*}


\subsection{A new look at Hadwiger invariants}
\label{ss:newlook}
Let $\hat\Ac$ be a set of rational hyperplanes in $\R^n$ that is stable under $\Z^n$-translations
 and such that some subset of $\hat\Ac$ has singleton intersection.
Set $\Ac$ to be the image of $\hat\Ac$ under the natural projection from $\R^n$ onto $\R^n/\Z^n$;
 we assume that all the notions introduced above and the
 corresponding notation ($\Vc$, $\Vc_k$, $\Fc_k$, $E^k$, etc.) refer to this choice of $\Ac$.
Any simple $\Z^n$-valuation on $\hat\Ac$-polytopes can be viewed as
 a linear functional on $\Vc$.
Given an oriented flag $\Lc\in\Lambda^\Oc_k$, let $\pi^\Lc$ be the
 linear projection from $E^k$ to the space of convex chains on $[\Lc]$.
Writing $\vol_{[\Lc]}$ for the volume valuation on $[\Lc]$, we can now represent
 the Hadwiger invariant $\Upsilon_\Lc$ as 
$$
\Upsilon_\Lc = \vol_{[\Lc]}\circ \pi^\Lc\circ D^k.
$$
It follows that $\Upsilon_\Lc$ vanishes on $\Vc_{<k}$, hence defines
 a linear functional $\tilde\Upsilon_\Lc$ on $\Vc_k$.
Under the identification between $\Vc_k$ and a subspace of $\Fc_k$,
 this functional $\tilde\Upsilon_\Lc$ identifies with $\vol_{[\Lc]}\circ\pi^\Lc$.

Consider the function $\1_\Lc\in\Fc_k$ on $\Lambda^\Oc_k$ defined as follows:
 it takes value $1$ on $\Lc$, value $-1$ on the same flag with the opposite orientation,
 and value $0$ on all other flags from $\Lambda^\Oc_k$.
As $\Lc$ varies, elements $\1_\Lc$ form a generating set of $\Fc_k$,
 that is, every element of $\Fc_k$ can be written as a finite linear combination of $\1_\Lc$.
Due to our normalization of the volume $\vol_{[\Lc]}$, the functional $\tilde\Upsilon_\Lc$
 takes values $\pm 1$ on $\pm\1_\Lc$ and value 0 on all other generators of $\Fc_k$.
Hence, any linear functional $\nu$ on $\Fc_k$ can be written as
$$
\nu=\frac 12\sum_{\Lc\in\Lambda^\Oc_k} \nu(\1_\Lc) \tilde\Upsilon_\Lc.
$$
Indeed, when evaluated on $\1_\Lc$ for $\Lc\in\Lambda^\Oc_k$, the left-hand side yields $\nu(\1_\Lc)$,
 while the right-hand side equals
$$
\frac 12\left(\nu(\1_\Lc)\tilde\Upsilon_\Lc(\1_\Lc) + \nu(\1_{-\Lc})\tilde\Upsilon_{-\Lc}(\1_\Lc)\right)=\nu(\1_\Lc),
$$
 where $-\Lc\in\Lambda^\Oc_k$ is, geometrically, the same flag as $\Lc$ but with the opposite orientation.
Here, we used that $\tilde{\Upsilon}_\Lc(\1_\Lc)=1$ and that $\tilde\Upsilon_{-\Lc}(\1_\Lc)=-1$;
 also, we used that $\1_{-\Lc}=-\1_\Lc$.

\subsection{Hadwiger invariants generate $\Vc^*$}
\label{ss:gen-pf}
We now show that $\Vc^*$ is generated by the Hadwiger invariants $\Upsilon_\Lc$,
 which is a part of the Main Theorem.
Let $\mu\in\Vc^*$ be any simple $\Z^n$-valuation on $\hat\Ac$-polytopes.
One can view $\mu$ as a linear functional on $\Vc$.
It suffices to show that, if $\mu$ vanishes on $\Vc_{<k}$, then
 there is a (possibly infinite) linear combination $\mu'$ of rank $k$ Hadwiger invariants
 such that $\mu-\mu'$ vanishes on $\Vc_{\le k}$.
Note that any rank $k$ Hadwiger invariant vanishes on $\Vc_{<k}$,
 hence it defines a linear functional on $\Vc_k$.
Both $\mu$ and $\mu'$ are now viewed as linear functionals on $\Vc_k$.
Thus, it suffices to find a linear combination $\mu'$ of rank $k$ Hadwiger invariants that
 takes the same values on $\Vc_k$ as $\mu$.
By \ref{ss:newlook}, the desired $\mu'$ is
$$
\mu':=\frac 12\sum_{\Lc\in\Lambda^\Oc_k} \mu(\1_\Lc) \Upsilon_\Lc.
$$

\subsection{Distinct ranks are independent}
\label{ss:rankk}
Hadwiger invariants of distinct ranks are linearly independent.

\begin{lem*}
Consider a linear relation $\sum_{\Lc} f(\Lc)\Upsilon_\Lc=0$ between Hadwiger invariants,
 where $f$ is an alternating function on $\bigcup_k\Lambda^\Oc_k$.
Then
$$
\sum_{\Lc\in\Lambda^\Oc_k} f(\Lc)\Upsilon_\Lc=0,\qquad\forall\ k=0,\dots,n.
$$
\end{lem*}

Unfortunately, this result does not generalize to wider topological setups considered in this paper.
Analogs of it appear to be false in general.
Hence, the method of the proof is also very specific to the setting of
 rational hyperplanes and $\Z^n$-invariance; this method is classical.

\begin{proof}
For any positive integer $\la$, consider the map $\la\times:\R^n\to \R^n$ that
 multiplies all vectors by $\la$.
By our assumptions on the set $\hat\Ac$ of rational hyperplanes,
 one has $\la\times H\in\hat\Ac$ for every $H\in\hat\Ac$,
 where $\la\times H$ is the image of $H$ under the map $\la\times$.
Precomposing with $\la\times$ produces, therefore, a $\Z^n$-valuation on $\hat\Ac$ out of any
 $\Z^n$-valuation on $\hat\Ac$.
Observe now that Hadwiger invariants of different ranks have different scaling properties
 with respect to $\la\times$, namely, $\Upsilon_\Lc\circ (\la\times)=\la^k\Upsilon_\Lc$,
 for any rank $k$ Hardwiger invariant $\Upsilon_\Lc$.
It now remains to look at the identity
$$
0=\sum_{k=0}^d\sum_{\Lc\in\Lambda^\Oc_k} f(\Lc)\Upsilon_\Lc\circ(\la\times)=
\sum_{k=0}^d \la^k\sum_{\Lc\in\Lambda^\Oc_k} f(\Lc) \Upsilon_\Lc
$$
 as the coincidence of two polynomial functions of $\la$, and equate
 the coefficients with the powers $\la^k$ separately for each $k$.
\end{proof}

\section{Linear algebra}
\label{s:linalg}
In this section, we study relations and second syzygies (relations between relations)
 in a context, where infinite linear combinations are allowed.

\subsection{Dual function spaces}
\label{ss:dual-fs}
Let $\Omega$ be an infinite set, and let $\Q[\Omega]$ stand for the
 rational vector space of all finitely supported functions on $\Omega$.
Recall that a function $\al:\Omega\to\Q$ is \emph{finitely supported} if
 the \emph{support} $\supp(\al):=\{x\in\Omega\mid \al(x)\ne 0\}$ of $\al$ is finite.
As is well known, the dual space of $\Q[\Omega]$ is the space $\Q^\Omega$ of all functions on $\Omega$,
 and the canonical pairing between $\Q^\Omega$ and $\Q[\Omega]$ can be written as
$$
\<f,\al\>:=\sum_{x\in\Omega} f(x)\al(x),\quad f\in\Q^\Omega,\ \al\in\Q[\Omega].
$$
Even though, formally, the sum in the right-hand side is infinite,
 it has only finitely many nonzero terms, by the assumption that the support of $\al$ is finite.

\subsection{Relations}
\label{ss:rels}
Now consider a vector subspace $\Uc\subset \Q[\Omega]$.
A family $(f_l)_{l\in\lf}$ of functions $f_l\in \Q^\Omega$ indexed by an arbitrary set $\lf$
 (possibly uncountable) is said to be a \emph{defining family of relations} for $\Uc$ if
$$
\Uc=\{\al\in\Q[\Omega]\mid \<f_l,\al\>=0\quad \forall l\in\lf\}.
$$
More generally, a \emph{relation} for $\Uc$ is any function $f\in\Q^\Omega$
 such that $\<f,\al\>=0$ for all $\al\in\Uc$.
Let $\Uc^\perp$ stand for the vector space of all relations for $\Uc$.
We want to describe $\Uc^\perp$ in terms of defining relations $f_l$.
Say that a family $(f_l)$ is \emph{locally finite} if, for every $x\in\Omega$,
 the set of all $l\in\lf$ such that $f_l(x)\ne 0$ is finite.

\begin{thm*}
If $\Uc$ is a subspace of $\Q[\Omega]$, and $(f_l)_{l\in\lf}$ is a locally finite defining family
 of relations for $\Uc$, then any $f\in\Uc^\perp$ can be written as
 $$
 f=\sum_{l\in\lf} \la_l f_l,\quad \la_l\in\Q
 $$
\end{thm*}

Here, the summation is over the possibly infinite set $\lf$;
 it makes sense since, for every $x\in\Omega$, only finitely many terms
 in the sum for $f(x)$ are nonzero.
Proof of the above theorem is 
 given in later sections.

\subsection{Combining relations}
\label{ss:combi-rel}
In one direction, the theorem stated in Section \ref{ss:rels} is straightforward:

\begin{lem*}
Let $\Uc$ and $f_l$ be as in \ref{ss:rels}.
Any (possibly infinite) linear combination $\sum_{l\in\lf} \la_l f_l$ with rational
 coefficients $\la_l$ is a relation for $\Uc$.
\end{lem*}

\begin{proof}
Set $f=\sum_l \la_l f_l$, where $l\mapsto \la_l$ is any function from $\lf$ to $\Q$.
Take any $\al\in\Uc$; we want to show that $\<f,\al\>=0$.
Indeed, 
$$
\<f,\al\>=\sum_{x\in\Omega} \al(x)\sum_{l\in\lf} \la_l f_l(x).
$$
Observe that this is a finite sum.
Firstly, let $\supp(\al)$ be the support of $\al$,
 which is a finite subset of $\Omega$; secondly, let $\lf_x$, for any $x\in\supp(\al)$,
 be the set of all $l\in\lf$ with $f_l(x)\ne 0$.
Now, the first sum in the right-hand side can be rewritten as a sum over $x\in \supp(\al)$,
 while the second sum as a sum over $l\in\lf_x$.
Set $\lf_\al$ to be the union of all $\lf_x$ over $x\in\supp(\al)$; then
 the sum in the right-hand side can be understood as a sum over the
 finite direct product $\supp(\al)\times \lf_\al$.
Changing the order of summation, we obtain that
$$
\<f,\al\>=\sum_{l\in\lf_\al} \la_l \sum_{x\in\supp(\al)} \al(x)f_l(x)=\sum_{l\in\lf_\al} \la_l \<f_l,\al\>=0.
$$
Hence, $f$ is a relation for $\Uc$, as claimed.
\end{proof}

\subsection{A compactness property}
\label{ss:comp}
Define a \emph{finite affine functional} on $\Q^\Omega$ as an affine functional
 $\xi:\Q^\Omega\to\Q$ of the form $\xi(f):=\<f,\al\>+b$, where $\al\in\Q[\Omega]$ and $b\in\Q$.
Now let $\Sc$ be any family of finite affine functionals on $\Q^\Omega$.
Say that $\Sc$ is \emph{finitely consistent} if any finite subfamily $\Sc'\subset\Sc$
 defines a nonempty affine subspace
$$
L_{\Sc'}:=\{f\in\Q^\Omega\mid \xi(f)=0\quad \forall \xi\in\Sc'\}.
$$
The following is a version of the Tychonoff compactness argument.

\begin{thm*}
Let $\Sc$ be a finitely consistent family of finite affine functionals on $\Q^\Omega$.
Then $\Sc$ is consistent: there is an element $g\in\Q^\Omega$ such that $\xi(g)=0$ for all $\xi\in\Sc$.
\end{thm*}

\begin{proof}
Fix $x\in\Omega$.
Since $\Sc$ is finitely consistent, there is a rational number $q_x\in\Q$ with the property that
 $\Sc\cup\{f\mapsto f(x)-q_x\}$ is also finitely consistent.
Indeed, the intersection of $\{f(x)\mid f\in L_{\Sc'}\}\subset\Q$ over all finite subcollections $\Sc'\subset \Sc$
 is nonempty, since the intersection of any finite subset of terms in this intersection is nonempty,
 and every term is either $\Q$ or a single point.
Hence, either $q_x$ is uniquely defined, or it can be chosen arbitrarily;
 in the latter case, one can always set $q_x:=0$.
Add the affine functional $\xi_x:f\mapsto f(x)-q_x$ to $\Sc$ and proceed with ``the next'' point.
By transitive induction (or by the Zorn lemma), $\Sc$ can be enlarged in this way
 to include, for every $x\in\Omega$, some functional of the form $\xi_x$;
 moreover, the enlarged family of affine functionals remains finitely consistent.
One may now assume that $\Sc$ contains all $\xi_x$.

Let $g$ be the function $x\mapsto q_x:=-\xi_x(0)$ (here, 0 denotes the neutral element of
 the additive group $\Q^\Omega$, that is, the zero function on $\Omega$).
We claim that $g$ satisfies $\xi(g)=0$ for all $\xi\in\Sc$.
Each individual $\xi\in\Sc$ has the form $\xi(f)=\<f,\al\>+b$ for some $\al\in\Q[\Omega]$ and $b\in\Q$.
Denote by $x_1$, $\dots$, $x_n$ all points of $\supp(\al)$, and set $q_i:=q_{x_i}$ for $i=1$, $\dots$, $n$;
 also, set $\xi_i:=\xi_{x_i}$.
Since the finite subfamily $\{\xi,\xi_1,\dots,\xi_n\}$ of $\Sc$ is consistent,
 there is a function $g'\in\Q^\Omega$ with $\xi(g')=0$ and $\xi_i(g')=0$ for all $i=1$, $\dots$, $n$.
It follows from $\xi_i(g')=0$ that $g'(x_i)=g(x_i)$, therefore,
$$
\xi(g)=\sum_{i=1}^{n} g(x_i)\al(x_i)+b=\sum_{i=1}^{n} g'(x_i)\al(x_i)+b=0,
$$
 where the last equality is the same as $\xi(g')=0$.
Thus, indeed, $g\in\Q^\Omega$ satisfies $\xi(g)=0$ for all $\xi\in\Sc$, as claimed.
\end{proof}

\subsection{Proof of the theorem from \ref{ss:rels}}
\label{ss:rels-proof}
In one direction, this theorem was already proved in Section \ref{ss:combi-rel}.
Suppose now that $f\in\Uc^\perp$, and we want to represent $f$ as $\sum \la_l f_l$,
 where the sum is over $l\in\lf$, and $\la_l$ are some rational coefficients.
View the system of yet unknown coefficients $\la_l$ as a function $\la:l\mapsto \la_l$ on $\lf$,
 that is, as an element of $\Q^\lf$.
Let $\Sc$ be the system of equations $\xi_x(\la)=0$, where $\xi_x(\la):=f(x)-\sum_l \la_lf_l(x)$.
Note that, in the right-hand side of the formula for $\xi_x$, the sum is in fact finite:
 it includes only finitely many $l\in\lf$ such that $f_l(x)\ne 0$.
Therefore, $\xi_x$ is a finite affine functional on $\Q^\lf$.

\begin{lem*}
  The system $\Sc$ is finitely consistent.
\end{lem*}

\begin{proof}
Let $\Sc'=\{\xi_{x_1},\dots,\xi_{x_n}\}\subset\Sc$ be a finite subcollection,
 and set $\Omega':=\{x_1,\dots,x_n\}$.
Consistency of $\Sc'$ means the following:
 the restriction $f':=f|_{\Omega'}$ of $f$ to $\Omega'$ is a linear combination
 of the restrictions $f'_l:=f_l|_{\Omega'}$.
To this end, it suffices to show that $f'$ vanishes on all $\al\in\Q[\Omega']$
 such that $\<f'_l,\al\>=0$ for all $l\in\lf$
 (we use the finite dimensional analog of the theorem from Section \ref{ss:rels},
 which is a classical fact of linear algebra).
Indeed, each such $\al$ lies in $\Uc$, and $f$ lies in $\Uc^\perp$.
\end{proof}

By the compactness property, see \ref{ss:comp},
 there is a function $\la\in\Q^\lf$ such that $\xi_x(\la)=0$ for all $x\in\Omega$,
 that is, $f(x)$ equals to $\sum_l \la_lf_l(x)$.
Since this is true for all $x$, we obtain that $f=\sum \la_lf_l$.

\subsection{A proof of the main theorem assuming \ref{ss:descVck}}
Assuming the theorem stated in \ref{ss:descVck}, we can now prove the main theorem.
It remains only to prove that any relation between the Hadwiger invariants is
 a (possibly infinite) linear combination of the reciprocity laws and period vanishing conditions.
This follows from \ref{ss:descVck} and the theorem of \ref{ss:rels}.
Firstly, we may fix $k$ and study only the relations between rank $k$ Hadwiger invariants, by \ref{ss:rankk}.
Set $\Uc:=\Vc_k\subset \Fc_k$; in this context, $\Omega$ is the set $\Lambda^\Oc_k$ of all oriented flags from $\Lambda_k$,
 cf. Section \ref{ss:degfilt0}.
Reciprocity laws and period vanishing conditions are then interpreted
 as linear equations $f_l$ defining $\Uc$, together with the relations defining
 alternating functions on $\Omega$.
Here the indices $l$ can be understood as flags of the form $\Lc\sm\{M\}$, where $\Lc\in\Lambda_k$
 and $M\in\Lc$; the index set $\lf$ is therefore the set of all such flags.

Let us check that the family $(f_l)$ is locally finite, that is, there are only finitely many
 reciprocity laws and period vanishing conditions involving a given flag $\Lc\in\Lambda_k$.
Indeed, such reciprocity laws or period vanishing conditions correspond bijectively
 to flags of the form $\Lc\sm\{M\}$, where $M\in\Lc$ has codimension $>0$; the number of such flags
 is $k$, i.e., the cardinality $k+1$ of $\Lc$ minus one.
The fact that the common zero locus of all $f_l$ is indeed $\Uc$ follows from \ref{ss:descVck}.
Now, \ref{ss:rels} implies that any relation between Hadwiger invariants, that is,
 any linear functional vanishing in $\Uc$, is a linear combination of $f_l$s, as desired.

\section{Semiregular CW spaces}
\label{s:semireg}
In this section, we describe a special type of CW spaces,
 which generalize regular cellulations.

\subsection{Atoms}
\label{ss:atom}
Consider a topological space $X$ and a finite collection $\Ac$ of its subsets.
Given any subcollection $\Ac'\subset\Ac$, let $X_{\Ac'}$ be the set of all points $x\in X$
 with the property that $x\in H$ for each $H\in\Ac'$ and $x\notin H$ for each $H\in\Ac\sm \Ac'$.
\emph{Atoms} of $\Ac$ are defined as components of $X_{\Ac'}$, for all possible subcollections $\Ac'\subset\Ac$.
By definition, $X$ is a disjoint union of its atoms.
Section \ref{s:top-arr} considers a setting, in which members of $\Ac$ have some properties
 of affine hyperplanes, and atoms have some properties of convex polytopes.

\subsection{CW spaces}
\label{ss:CW}
Recall (see, e.g., \cite{Hat02}) that a \emph{finite CW space} (or CW complex) is a structure on a
 topological space $X$ that includes, firstly, a representation of $X$
 as a union of disjoint subspaces called \emph{cells} and, secondly,
 a choice, for every cell $e\subset X$, of a so-called \emph{characteristic map} $\varphi_e:\ol\D^k\to \ol e\subset X$.
Here, $\ol\D^k$ denotes the closed unit disk in $\R^k$, i.e., the closure
 of the open unit disk $\D^k$, and $k$ may depend on $e$;
 it is called the \emph{dimension} of $e$ and denoted by $\dim(e)$.
Characteristic maps are assumed to be continuous.
Also, one requires that the restriction $\varphi_e:\D^k\to e$ of $\varphi_e$ to $\D^k$ be a homeomorphism and that
 $\varphi_e(\Sb^{k-1})$ be contained in the union of cells of dimensions $<k$.
In general, $\varphi_e$ does not have to be one-to-one on $\Sb^{k-1}$,
 neither has it to cover certain cells entirely.
\emph{Regular CW spaces} are, by definition, CW spaces such that all the
 characteristic maps are homeomorphisms, and the boundary of every cell is a union of cells.
We will only consider finite CW spaces; for this reason, we omit the word ``finite''.
Finite CW spaces are compact and metrizable (in particular, Hausdorff).

Cells of dimension $k$ are also called \emph{$k$-cells}.
Define a \emph{CW subspace} of $X$, or a \emph{subcomplex} of $X$, as
 a closed subspace represented as a union of cells.
\emph{Normal CW spaces} are those for which the closure, equivalently, the boundary of every cell is a subcomplex.
The \emph{dimension} $\dim(X)$ of $X$ is defined as the maximum of the dimensions
 of all cells of $X$.

\subsection{Drawings in the plane}
\label{ss:draw}
An important example of a regular CW space of dimension two is as follows.
Let $\Ac$ be a finite collection of curves in $\R^2$, each curve being a closed subset homeomorphic to $\R$.
Suppose that any pair of distinct curves $L_1$, $L_2$ from $\Ac$ intersect
 \emph{transversely in the topological sense}, that is, every point $q\in L_1\cap L_2$
 has an open neighborhood $U$ in $\R^2$ such that $U$ is homeomorphic to $\D^2$
 via a homeomorphism that sends $U\cap L_1$ and $U\cap L_2$ to two distinct diameters of $\D^2$.
Assume also that there are only finitely many points that belong to
 more than one curve from $\Ac$.
If these assumptions are fulfilled, $\Ac$ is called a \emph{regular drawing} in $\R^2$.
With $\Ac$, one associates a regular CW space $X$ as follows.
Viewed as a topological space, $X:=\R^2\cup\{\infty\}$ is the one-point compactification of $\R^2$.
Cells of $X$ are defined as atoms of $\Ac$; it is not hard to check that
 $X$ thus obtained is indeed a regular CW space.

\subsection{Semiregular CW spaces}
\label{ss:srCW}
The following notion generalizes regular CW spaces as well as $\Delta$-complexes of \cite[Section 2.1]{Hat02}.
Let $X$ be a normal CW space.
Suppose that, for every cell $e$ of $X$ and any cell $e'\subset \ol{e}$, there is a homeomorphism
 $\psi_{e,e'}$, called a \emph{transition homeomorphism}, between $\ol\D^{\dim(e')}$ and
 a subset of $\ol\D^{\dim(e)}$ such that $\varphi_{e'}=\varphi_e\circ\psi_{e,e'}$.
It follows that the $\psi_{e,e'}$-image of $\ol\D^{\dim(e')}$ lies in the boundary sphere of $\ol\D^{\dim(e)}$,
 unless $e'=e$.
Under these assumptions, $X$ is called a \emph{semiregular CW space}.
By definition, any semiregular CW space is normal.

Consider a cell $e\subset X$ of a semiregular CW space $X$,
 together with the corresponding characteristic map $\varphi_e:\ol\D^k\to\ol e$.
\emph{Pullbacks}, i.e., components of the full preimage under $\varphi_e$, of
 cells in $\ol e$ form a partition of $\ol\D^k$; this is a structure of a regular
 CW space on $\ol\D^k$, where the transition homeomorphisms play a role of characteristic maps.
Write $\ol\D_e$ for the closed disk equipped with this regular CW structure.
Note that $\varphi_e:\ol\D_e\to\ol e$ becomes a cellular map, i.e., it takes
 cells to cells, moreover, every cell is mapped forward homeomorphically.
However, distinct cells of $\ol\D_e$ may map onto the same cell of $X$.

\subsection{Example: 2-torus}
\label{ex:T2-0}
One can take a convex polytope $P$ and identify some pairs of its faces using, say, affine homeomorphisms.
Then the quotient space has a natural structure of a semiregular CW space.
For example, representing the torus $\T^2:=\R^2/\Z^2$
 as the quotient space of the unit
 square $[0;1]\times [0;1]$ by $(x,0)\sim (x,1)$ and $(0,y)\sim (1,y)$,
 one equips $\T^2$ with a structure of a semiregular CW space.
Cells are: the interior of the square, the relative interiors of the two edges
 (any pair of parallel edges of $[0;1]^2$ gives rise to a single 1-cell of $\T^2$), and the point $(0,0)$.
Note that this CW space is not regular.

\subsection{Poset of cells}
With any CW space $X$, associate its \emph{poset of cells}. 
Elements of this poset (=partially ordered set) are cells of $X$, and the strict order $\prec$ on this poset
 is such that $e_1\prec e_2$ for two cells $e_1$, $e_2$ if and only if $\ol e_2\subsetneq\ol e_1$
 (note the reversed inclusion: $e_1\prec e_2$ means that $e_2$ is geometrically smaller than $e_1$).
Recall that an \emph{order chain} in a poset is just a linearly ordered subset.
For the posets of cells in $X$, an order chain
 can be written as $e_{k_0}\prec\dots\prec e_{k_l}$ so that
 the indices $k_i$ are equal to the codimensions of the corresponding cells $e_{k_i}$,
 in which case they are strictly increasing: $k_0<\dots < k_l$.

\subsection{Barycentric subdivision}
Go back to the case where $X$ is a semiregular CW space.
For every cell $e$ of $X$, set $c_e=\varphi_e(0)$ and call this point the \emph{center of $e$}.
Also, a \emph{radial interval} in $e$ is defined as $\{\varphi_e(tv)\mid t\in (0;1)\}$ for some $v$
 in the boundary sphere of $\D^{\dim(e)}$.
Now, it makes sense to talk about the \emph{barycentric subdivision} of $X$.
Relatively open simplices of the barycentric subdivision can be defined by induction on their dimension as follows.
Simplices of dimension 0 are the centers of all cells.
Assuming by induction that all $k$-simplices, i.e., simplices of dimension $k$, are defined,
 consider a cell $e$ (of any dimension $>k$) and a $k$-simplex $s\subset \ol e\sm e$.
Connect $c_e$ with all points of $s$ by radial intervals.
The union 
 of these intervals is by definition a $(k+1)$-simplex of the barycentric subdivision.

It is straightforward to verify that $X$ is partitioned into simplices
 of the barycentric subdivision; equipped with this partition, $X$ is denoted by $X_s$.
We will therefore refer to the simplices of the barycentric subdivision as \emph{simplices of $X_s$}.
Consider a cell $e$ of $X$ and the corresponding characteristic map $\varphi_e:\ol\D^k\to \ol e$.
Note that $\varphi_e$ takes the barycentric subdivision of $\ol\D_e$ to
 the barycentric subdivision of $\ol e$.
Restricting $\varphi_e$ to closed simplices of $(\ol\D_e)_s$,
 one obtains characteristic maps of some new CW structure on $X$, whose cells are simplices of $X_s$.
This is, in fact, a regular CW structure corresponding to a geometric simplicial complex.

\subsection{Example: 2-torus bis}
\label{ex:T2}
Every simplex of $X_s$, for a semiregular CW space $X$,
 corresponds to a unique order chain in the poset of cells, that is,
 cells whose centers form the given simplex are linearly ordered.
However, different simplices may give rise to the same order chain,
 as the following example shows.

Let $X=\T^2$ be the real 2-torus represented as in Example \ref{ex:T2-0}.
Observe that there are four distinct 2-simplices of $X_s$ associated
 with the same order chain, namely, $e_0\prec e_1\prec e_2$,
 where $e_0$ is the only 2-cell, $e_1$ is the horizontal edge, and $e_2$ is the only 0-cell.
Also, there are two simplices associated with the order chain $e_0\prec e_1$,
 two simplices associated with $e_1\prec e_2$, and four simplices associated with $e_0\prec e_2$.

\subsection{Base vertex}
Write $S=S(X)$ for the set of all simplices of $X_s$.
Recall that any simplex $s\in S$ corresponds to some order chain $e_{k_0}\prec\dots\prec e_{k_l}$ in
 the poset of cells.
The center of $e_{k_l}$ is then called the \emph{base vertex of $s$}.
Let $S^{k_l}_l$ be the subset of all $s$ with given $l$ and $k_l$.
In other words, $S^k_l$ consists of all dimension $l$ simplices in $X$
 whose base vertex lies in a codimension $k$ cell.
Given any simplex $s\in S$, let $\Gamma s$ be the facet (=codimension 1 face)
 of $s$ that includes all vertices of $s$ except only the base vertex.

\subsection{Dual cells and stars}
\label{ss:dualcell}
Let $X$ be a semiregular CW space.
Also, let $e$ be a cell of $X$ with center $c_e$.
Define the \emph{dual cell} $e^\star$ of $e$ as the union of all relatively open simplices of $X$
 associated with order chains $e_{k_0}\prec\dots\prec e_{k_l}$ such that $e_{k_l}=e$.
Sets of the form $e^\star$, where $e$ is a cell of $X$, are called \emph{dual cells of $X$}.
If all dual cells are indeed homeomorphic to open disks, then we say that $X$ is \emph{nonsingular}.
Also, define the \emph{star} $\St(e)$ as the union of all relatively open simplices of $X$ incident to
 the center of $e$.
For a nonsingular $X$, stars of cells are also homeomorphic to open disks.

Every connected nonsingular semiregular CW space $X$ is a topological manifold,
 with stars of all cells forming an atlas.
In particular, the dimension $\dim(X)$ of $X$ is well defined.
If $\dim(X)=n$, 
 then every cell of $X$ lies in the closure of some $n$-cell.

\subsection{Cellular chain spaces}
\label{ss:cell-chain}
We now assume that $X$ is a nonsingular semiregular CW space of dimension $n$.
It is straightforward that dual cells define a CW structure on $X$ called the \emph{dual cell structure}.
Viewed as a CW space with this structure, $X$ will be denoted by $X_\star$.
Write $C_k=C_k(X_\star;\Q)$ for the space of all dimension $k$ cellular chains of $X_\star$
 over the rational numbers,
 and $C_\bullet$ for the direct sum of $C_k$ over all $k$ from $0$ to $n$.
In general, replacing an index by a bullet will mean the direct summation with respect to this index;
 the direct sum comes equipped with the grading that recovers the summands.
Recall that there is a natural boundary operator $\d:C_k\to C_{k-1}$
 for $k>0$, which makes $C_\bullet$ into an algebraic chain complex.
Homology spaces of this complex coincide with homology spaces of $X$.

Generators of $C_k$ can be identified with oriented dual $k$-cells of $X$,
 that is, oriented $k$-cells of $X_\star$.
Note that any simplex $s\in S^k_k$ defines not only a dual cell $e[s]$ of codimension $k$
 but also some orientation of this cell.
If $s$ and $s'$ both belong to $S^k_k$, share the same base vertex, and
 have exactly $k$ vertices in common, then $e[s]+e[s']=0$ in $C_k$,
 that is, the two simplices define two opposite orientations of the same dual cell.
We will write generators of $C_k$ as $e[s]$, assuming the relations just described.

In terms of these generators, the boundary operator simply drops the base vertex,
 i.e., takes the facet spanned by all other vertices.
One has $\d e[s]=\sum_{s'}\pm e[\Gamma s']$, where the sum is over all
 $s'$ in the same dual cell, that is, with the same base vertex as $s$,
 and the sign is plus or minus depending on whether $e[s']=e[s]$ or $e[s']=-e[s]$.
Note that, in this description of the boundary operator, the identity $\d^2=0$ follows
 from the simple observation that, whenever $\Gamma^2 s=\Gamma^2 s'$ for two
 $k$-simplices $s$, $s'$ in the same $k$-cell, one has $e[s]=-e[s']$, so
 that $e[\Gamma s]$ and $e[\Gamma s']$ enter the expression for $\d e[s]$
 with the opposite signs.

\section{Topological arrangements}
\label{s:top-arr}

Below, we define an additional structure on a semiregular CW space that generalizes affine hyperplane arrangements
 as well as toric arrangements.
It also includes regular drawings in the plane as described in \ref{ss:draw}.

\subsection{Topological arrangements}
\label{ss:top-arr}
Let $X$ be a nonsingular connected semiregular CW space, and let $\Ac$ be a finite set of its subcomplexes with
 the following properties:
\begin{enumerate}
  \item atoms of $\Ac$ are precisely the cells of $X$;
  \item every component of $H_1\cap\dots\cap H_k$, where $H_1,\dots,H_k\in\Ac$,
   is nonsingular of dimension at least $n-k$, where $n:=\dim(X)$.
\end{enumerate}
Assuming that these properties hold, the pair $(X,\Ac)$ is called a \emph{topological arrangement}.
Sometimes, we will also refer to $X$ and/or to $\Ac$ as a topological arrangement.
Members of $\Ac$ are called \emph{hyperplanes} of $\Ac$ (or of $(X,\Ac)$, or of $X$).
Components of $H_1\cap\dots\cap H_k$, where $H_1,\dots,H_k\in\Ac$, are called \emph{flats} of $\Ac$.

\subsection{Dual cells of flats}
\label{ss:C^L}
Consider any flat $L$ of $\Ac$ and let $L_\star$ be the CW space formed by
 \emph{dual cells of $L$}, that is, by dual cells with respect to the cell structure of $L$ induced by $\Ac|_L$.
Note that, for $L\ne X$, dual cells of $L$ are \emph{not} dual cells of $X$.
On the other hand, dual cells of $L$ can be understood as the intersections of
 dual cells of $X$ with $L$.
Moreover, there is a natural one-to-one correspondence between the cells of $X_\star$
 intersecting $L$ and the cells of $L_\star$; this correspondence is given by the intersection with $L$.

Given a flat $L$ of $\Ac$, define the linear map $\pi^L:C_k\to C_k$ by setting $\pi^L(e)=0$ or $e$
 depending on whether a dual $k$-cell $e$ intersects $L$  or not (in the latter case, the center of $e$ lies in $L$).
Then $\pi^L$ is a linear projector of $C_k$ onto its vector subspace $C^L_k:=\pi^L(C_k)$.
Note that $\d^L:=\pi^L\circ\d$ defines a boundary operator on $C^L_\bullet$, which
 makes the latter graded vector space into an algebraic chain complex.
Chains from $C^L_k$ can be interpreted as cellular $(k-l)$-chains in $L_\star$
 with coefficients in the coorientation sheaf $\Oc_L$ of $L$, where $l$ is the codimension of $L$.
Even though we will not always refer to this interpretation explicitly, it will be useful to
 view elements of $C^L_k$ as $(k-l)$-chains rather than as $k$-chains.
Note that $C_k$ can be interpreted as $C^X_k$.

\subsection{Flags}
\label{ss:flags}
By a \emph{flag} of $X$ (or of $\Ac$, or of $(X,\Ac)$), we mean a linearly ordered
 by inclusion set of flats.
Hence, any flag $\Lc$ of $X$ has the form $L^{k_0}\supset \dots\supset L^{k_l}$, where
 $k_0<\dots<k_l$ are the codimensions of the flats $L^{k_0}$, $\dots$, $L^{k_l}$, respectively.
Set $[\Lc]:=L^{k_l}$ and call it the \emph{base flat} of $\Lc$.
Write $\Lambda_l$ for the set of all flags of the form $L^0\supset\dots\supset L^l$,
 containing flats of all codimensions from 0 to $l$ inclusive.
Given a flag $\Lc$, define the \emph{truncation} $\Gamma\Lc$ of $\Lc$
 as the flag obtained by dropping the base flat from the flag $\Lc$.
If $\Lc\in\Lambda_l$, then $\Gamma\Lc\in\Lambda_{l-1}$.
In other terms, if $\Lc=(L^0\supset \dots\supset L^l)$, then $\Gamma\Lc=(L^0\supset\dots\supset L^{l-1})$.

\subsection{Chain complexes $C^{(l)}_\bullet$}
\label{ss:l-compl}
Define the rational vector space $C^{(l)}_k$ as the direct sum of $C^\Lc_k:=C^{[\Lc]}_k$ over all flags $\Lc\in\Lambda_l$.
Since there may be many flags $\Lc$ with the same base flat $L=[\Lc]$, the vector space
 $C^{(l)}_k$ includes many identical copies of $C^L_k$.
Recall that $C^L_\bullet$ is an algebraic chain complex, for every flat $L$ of $\Ac$,
 with the boundary operator $\d^L$.
For the chain complex $C^\Lc_\bullet$, which is by definition isomorphic to $C^{[\Lc]}_\bullet$,
 the boundary operator is denoted by $\d^\Lc$.
As a direct sum of chain complexes, $C^{(l)}_\bullet$ is also equipped with
 a structure of a chain complex; denote the corresponding boundary operator by $\d^{(l)}$.
Observe that dual $l$-cells intersect codimension $l$ flats at most by singletons.
Hence, $C^{(l)}_l$ should be informally understood as a certain space of 0-chains.
We write $\d$ instead of $\d^\Lc$ or $\d^{(l)}$ if the domain of $\d$ is clear from the context.
It is important to remember, however, that the action of $\d$ on $C^{(l)}_k$ is
 very different from the action of $\d$ on $C_k$.

\subsection{Labeling the generators}
\label{ss:labgen}
Generators of $C^{(l)}_k$ can be labeled by pairs $(e,\Lc)$ consisting of an
 oriented dual $k$-cell $e$ and a flag $\Lc\in\Lambda_l$ such that $e\cap [\Lc]\ne \0$.
The last condition implies that $k\ge l$: any dimension $k$ dual cell is
 disjoint from all codimension $l>k$ flats.
Denote by $[e,\Lc]$ the generator corresponding to a pair $(e,\Lc)$ as just described.
One can think of $C^{(l)}_k$ as the rational vector space generated by $[e,\Lc]$,
 the only relations being that reversing the orientation of $e$ means changing the sign of $[e,\Lc]$.
Informally, $C^{(l)}_k$ should be viewed as a space of $(k-l)$-chains.
Namely, the generator $[e,\Lc]$ gives rise to a dual cell $e\cap [\Lc]$ of $[\Lc]$,
 equipped with a flag, namely, $\Lc$, and a choice of an orientation of $e$.
The boundary operator acts on $[e,\Lc]$ as taking the boundary of $e\cap [\Lc]$,
 except that the coefficient system is different from $\Q$
 (an orientation of $e$ is not the same as an orientation of $e\cap [\Lc]$).
In fact, one can write
$$
C^{\Lc}_\bullet = C_\bullet([\Lc]_\star,\Oc_{[\Lc]}),
$$
 where $\Oc_L$ is the \emph{coorientation shief} (over $\Q$) of a flat $L$.
Over a small neighborhood of any point $q\in L$, the shief $\Oc_L$ is isomorphic
 to the constant shief $\Q$, but an isomorphism depends on the choice of
 a local coorientation of $L$ at $q$; reversing this coorientation means
 changing the sign of the isomorphism.

\subsection{An isomorphism}
\label{p:iso}
\emph{For any given $\Lc\in\Lambda_{l-1}$,
 the sum of $C^{\Lc'}_l$ over all $\Lc'\in\Lambda_{l}$ with $\Gamma\Lc'=\Lc$
 is naturally isomorphic to $C^\Lc_l$.
As a consequence, $C^{(l-1)}_l$ is naturally isomorphic to $C^{(l)}_l$.}

We will write $\iota_l:C^{(l-1)}_l\to C^{(l)}_l$ for this natural isomorphism.
Similarly to other notational conventions of this paper, we will sometimes omit
 the subscript $l$ of $\iota$ when the domain is clear from the context.

\begin{proof}
Fixing an orientation of each dimension $l$ dual cell $e$ that intersects $[\Lc]$ yields
 a basis of $C^\Lc_l$ as well as of $\bigoplus_{\Gamma\Lc'=\Lc} C^{\Lc'}_l$.
Note that every basis vector $[e,\Lc]$ of the former space yields a unique basis
 vector $[e,\Lc']$ of the latter space, where $\Lc'\supset\Lc$ is determined
 by the condition that $[\Lc']$ intersects $e$, equivalently, that the center of $e$ is contained in $[\Lc']$.
This correspondence $\iota:[e,\Lc]\mapsto [e,\Lc']$ between basis vectors is a bijection:
 $\Lc$ is recovered from $\Lc'$ as $\Lc=\Gamma\Lc'$.
It follows that $\iota$ extends to a unique linear isomorphism between
 the vector spaces $C^\Lc_l$ and $\bigoplus_{\Gamma\Lc'=\Lc} C^{\Lc'}_l$.
\end{proof}

The isomorphism $\iota$ allows to define the operator $\delta:C^{(l)}_l\to C^{(l-1)}_{l-1}$
 as $\d^{(l-1)}\circ\iota_l^{-1}$ (we will often write $\d\circ\iota^{-1}$, or even $\d\iota^{-1}$,
 if the domain of $\d$ is clear).
Note that $\delta$ does \emph{not} satisfy the relation $\delta\circ\delta=0$.

\subsection{Cochain spaces and coboundary maps}
\label{ss:coch}
Let $C^\bullet$ be the cochain complex that is dual to $C_\bullet$; write $d$ for the
 corresponding coboundary operator, i.e., the dual map of $\d$.
The natural nondegenerate bilinear pairing between $C^k$ and $C_k$ is denoted by
 $\<\cdot,\cdot\>$, so that $\<f,\al\>$ means the value of a cochain $f\in C^k$ on a chain $\al\in C_k$.
Similarly, $C^\bullet_\Lc$ denotes the dual cochain complex to $C^\Lc_\bullet$; the same
 convention applies to $C^\bullet_{(l)}$ and $C^\bullet_L$.
Write $d_\Lc$, $d_{(l)}$, and $d_L$ for the corresponding coboundary operators.
If the domain of the coboundary operator is clear from the context, we may omit the subscript and simply write $d$.
However, it is important to remember that $d$ may mean very different maps
 depending on their domain.
By definition, the coboundary operators act as dual operators to the
 corresponding boundary maps, i.e., $\<df,\al\>:=\<f,\d\al\>$
 (here, $\d$ may also mean very different maps depending on their domain).
We adopt the convention with no signs; different sign conventions will lead to the same results.

To lighten the notation, from now on, we write $E_k$ and $E^k$ for $C^{(k)}_k$ and $C^k_{(k)}$, respectively.
This notation is consistent with the one introduced in \ref{ss:elem} in the case of a finite toric arrangement.
Finally, the dual operator of $\delta:E_{k+1}\to E_k$ is
 denoted by $D:E^k\to E^{k+1}$, or by $D_{(k)}$ if $k$ needs to be emphasized.
Since $\delta$ can be written as $\d\circ\iota^{-1}$, the dual map $D$ is equal to $\iota^{*-1}\circ d$,
 where $\iota^*$ is the dual of $\iota$, and $\iota^{*-1}$ is the inverse of $\iota^*$.
Given $\xi\in E^k$, we can describe the action of $D\xi$ on a generator
 $[e,\Lc]\in E_{k+1}$ as follows:
$$
D\xi[e,\Lc]=\xi[e^+,\Gamma\Lc]-\xi[e^-,\Gamma\Lc].
$$
Here $\Lc\in\Lambda_{k+1}$, and $e$ is a dual $(k+1)$-cell crossing the flat $[\Lc]$;
 the intersection $e\cap [\Gamma\Lc]$ is a topological arc (a dual 1-cell of $[\Gamma\Lc]_\star$)
 with endpoints $e^+\cap[\Gamma\Lc]$ and $e^-\cap [\Gamma\Lc]$;
 the dual $k$-cells $e^\pm$ on the boundary of $e$ are recovered by the last condition.
Finally, the orientations of $e^\pm$ are as in the boundary of $e$.


\section{Varchenko--Gelfand algebras}
\label{s:VGalg}
In this section, we describe a version of the Varchenko--Gelfand algebras for topological arrangements.

\subsection{Definition}
\label{def:VG}
Consider a topological arrangement $(X,\Ac)$, and let $\Vc=\Vc(X)$ be
 the algebra of all locally constant functions from $X\sm \bigcup\Ac$ to $\Q$.
The original definition of Gelfand and Varchenko \cite{VG87} dealt only with affine hyperplane arrangements.
Even though their definition and the algebra structure extend verbatim to topological arrangements,
 one needs an essential modification to define the \emph{degree filtration} on $\Vc$.
Indeed, the original approach of \cite{VG87} was based on a representation of elements of $\Vc$
 as polynomials in \emph{Heaviside functions}, and the latter do not make sense for general
 topological arrangements.
Another minor difference between $\Vc$ and the original algebra of \cite{VG87} is that
 we work over the rationals rather than over the integers.

There is a natural nondegenerate bilinear pairing $\<\cdot,\cdot\>$ between the spaces $\Vc$ and $C_0$;
 namely,
$$
\<f,\al\>:=\sum \al_x f(x),\quad \forall f\in\Vc,\ \forall \al=\sum_{x}\al_x x\in C_0.
$$
Here, $\al\in C_0$ is represented as a sum of dual 0-cells $x$ (=centers of $d$-cells of $X$)
 with coefficients $\al_x\in\Q$.
It follows that $\Vc$ can be viewed as a vector space dual of $C_0$, that is, $\Vc$ identifies with $C^0$.
For example, it makes sense to talk about the orthogonal complement $U^\perp$ (=the annihilator) of
 a subspace $U\subset C_0$; then $U^\perp$ is a subspace of $\Vc$.

\subsection{Degree filtration}
\label{ss:degfilt}
Define a subspace $\Vc_{\le k}\subset \Vc$ as the annihilator of $\delta^{k+1}E_{k+1}$,
 equivalently, as the kernel of $D^{k+1}:\Vc\to E^{k+1}$.
As $k$ grows, the spaces $\delta^{k+1}E_{k+1}$ become smaller, and 
  the spaces $\Vc_{\le k}$ become larger:
$$
\Vc_{\le 0}\subset \Vc_{\le 1}\subset \dots \subset \Vc_{\le n}=\Vc.
$$
Here, the last equality $\Vc_{\le n}=\Vc$ follows from the fact that $\delta^{n+1}E_{n+1}=0$;
 indeed, $E_{n+1}=0$.
Thus, the subspaces $\Vc_{\le k}$ form a filtration of $\Vc$ called the \emph{degree filtration}.
Define the \emph{quotients of the degree filtrations}, i.e., the quotient spaces
 $\Vc_k:=\Vc_{\le k}/\Vc_{<k}$, where $\Vc_{<k}$ means $\Vc_{\le k-1}$.
For $k=0$, set $\Vc_0:=\Vc_{\le 0}$; equivalently, one can set $\Vc_{<0}:=0$ in
 order to make the preceding definition applicable.
One can see a transparent analogy between the defining property $D^{k+1}f=0$
 of elements $f\in\Vc_{\le k}$ and the description of a degree $k$
 polynomial as a smooth function whose all derivatives of order $k+1$ vanish;
 here $D$ is viewed as a kind of discrete differential.

\subsection{Degree zero}
\label{ss:deg0}
\emph{The elements of $\Vc_0$ are all constant functions.
 Moreover, there is a natural isomorphism between $\Vc_0$ and the
 $0$-cohomology space of $X$.}
First note that $E_1$ is isomorphic to $C_1$, by \ref{p:iso}, and that $\delta E_1=\d C_1$.
Recall that the space $C_1$ is spanned by oriented dual 1-cells,
 while the space $E_1$ is spanned by pairs $[e,\Lc]$, where $e$ is an oriented dual 1-cell,
 and $\Lc\in\Lambda_k$ is a flag of the form $X=L^0\supset L^1$, where $e\cap L^1\ne\0$.
Since $\Lc$ is determined by $e$, one can indeed identify $e$ with $[e,\Lc]$,
 thus identifying $C_1$ with $E_1$.
By definition of the degree filtration, this says that $\Vc_0$ coincides with the space of $0$-cocycles,
 that is, there is a natural isomorphism between $\Vc_0$ and $\Hs^0(X)$, the 0-cohomology space of $X$.
Recall that $X$ is connected, therefore, $\Hs^0(X)=\Q$.
Equivalently, the $0$-cocycles are precisely the constant functions.

\subsection{Degree one}
\label{ss:deg1}
Let us also discuss the degree one term $\Vc_{1}$.
Note that $C^1$ is isomorphic to $E^1$ and that the kernel of $d:\Vc=E^0\to C^1$
 is precisely $\Vc_{\le 0}$ (recall that $d$ is the usual coboundary operator
 acting on $C^\bullet$, see \ref{ss:coch}, and that $E^0=C^0$).
Therefore, $d$ induces a linear embedding of $\Vc_1$ into $C^1$.
Given any $f\in\Vc$, the 1-cochain $df$ measures the jumps made by the function $f$ as
 its argument crosses a certain hyperplane of $\Ac$.
The 1-cochain $df$ can also be viewed as a 0-cochain $g_H$ defined on every
 hyperplane $H$ of $\Ac$, with coefficients in the coorientation sheaf $\Oc_H$ of $H$.
With this interpretation, the assumption $D^2 f=0$, which by definition is equivalent to $f\in\Vc_{\le 1}$,
 means that $g_H$ is a 0-cocycle, for each $H\in\Ac$.
As a consequence, one obtains the following isomorphic embedding
$$
\Vc_1\hookrightarrow\bigoplus\Hs^0(H,\Oc_H).
$$
Note that every term $\Hs^0(H,\Oc_H)$ in the right-hand side is isomorphic either to $\Q$, if $H$
 is coorientable, or to $0$, of $H$ is not coorientable.

Suppose --- this is an important special case --- that all hyperplanes of $\Ac$ are coorientable.
Then elements of $\Vc_1$ can be identified with \emph{alternating functions on cooriented hyperplanes of $\Ac$}.
More precisely, by a \emph{cooriented hyperplane} of $\Ac$ we mean a hyperplane from $\Ac$
 with a specific choice of a coorientation; in this sense, the same geometric hyperplane
 gives rise to two distinct cooriented hyperplanes that differ only by the coorientations.
A function on cooriented hyperplanes of $\Ac$ is said to be \emph{alternating} if its value
 changes sign whenever the hyperplane reverses its coorientation.

\subsection{Spaces $\Fc_k$}
\label{ss:Fck}
For every flat $L$ of $\Ac$ of codimension $k$, 
 recall that $C^L_k$ can be interpreted as the space of (cellular) 0-chains in $L_\star$ with coefficients in $\Oc_L$.
Similarly, $C^k_L$ is the space of 0-cochains in $L_\star$ with coefficients in $\Oc_L$.
Now fix a flag $\Lc\in\Lambda_k$ and recall that $C^k_\Lc$ is just a copy $C^k_L$, where $L=[\Lc]$.
Also, $C^k_\Lc$ can be viewed as a direct summand, hence a vector subspace, of $E^k$.
In particular, it makes sense to apply $D$ to elements of $C^k_\Lc$.
The condition $Df=0$ for $f\in C^k_\Lc$ means that the corresponding 0-cochain in $L_\star$ is a cocycle.
Hence, every $f\in C^k_\Lc$ with $Df=0$ defines a unique element of $\Hs^0(L;\Oc_L)$;
 moreover, this is a linear embedding.
Recall that $\Hs^0(L;\Oc_L)$ is isomorphic to $\Q$ if $L$ is coorientable and to $0$ otherwise.
Observe that, in the former case, an isomorphism is non-canonical --- it depends on
 a choice of a coorientation of $L$.
Define
$$
\Fc_k:=\bigoplus_{\Lc\in\Lambda_l} \Hs^0([\Lc];\Oc_{[\Lc]}).
$$
This is a vector space over $\Q$ that is naturally a subspace of $E^k$ given as $\ker[D:E^k\to E^{k+1}]$.
If all codimension $k$ flats are coorientable, then elements of $\Fc_k$ can be understood
 as alternating functions on oriented $k$-flags.
By an \emph{oriented $k$-flag}, we mean a flag $\Lc\in\Lambda_k$ equipped with some
 orientation of $\Lc$, that is, a coorientation of $[\Lc]$.
A function on oriented $k$-flags is \emph{alternating} if its value changes sign every time
 the flag reverses its orientation.
The same description is applicable to the general case, when not all $\Lc$s are orientable,
 if one understands that an alternating function on oriented flags is undefined on
 those flags which cannot be oriented.

Write $\Lambda^\Oc_k$ for the set of flags $\Lc\in\Lambda_k$ equipped with orientations.
Thus, any orientable flag from $\Lambda_k$ gives rise to two elements of $\Lambda^\Oc_k$
 that differ only by the orientations, while a non-orientable flag from $\Lambda_k$ defines
 no elements of $\Lambda^\Oc_k$ at all.
As we discussed earlier, $\Fc_k$ identifies with the space of all $\Q$-valued alternating functions on $\Lambda^\Oc_k$.

\subsection{Orientable arrangements}
\label{ss:ori}
A topological arrangement $(X,\Ac)$ is said to be \emph{orientable} if every flat is coorientable.
In an orientable arrangement, all flags are orientable, therefore,
 the description of the spaces $\Fc_k$ becomes the most straightforward.
The set $\Lambda^\Oc_k$ has in the orientable case twice as many elements as
 the set $\Lambda_k$, and the dimension of $\Fc_k$ over $\Q$ equals the cardinality of $\Lambda_k$.
For every flat $L$ of an orientable arrangement $\Ac$, the cohomology spaces
 $\Hs^k(L;\Oc_L)$ are isomorphic to $\Hs^k(L;\Q)$ but an isomorphism is defined only up to a sign.

\subsection{Cocycle condition}
\label{ss:cocycle}
Consider an element $\xi\in E^k$.
Say that $\xi$ \emph{satisfies the cocycle condition} if $d\iota^{*}\xi=0$.
Here, the isomorphism $\iota^{*}$ takes $\xi$ to $C^k_{(k-1)}$, and 
 $d$ takes $\iota^{*}\xi$ to $C^{k+1}_{(k-1)}$
 (recall that $\iota$ is defined in \ref{p:iso} and that $d:C^k_{(k-1)}\to C^{k+1}_{(k-1)}$ is
 the coboundary operator of the cochain complex $C^\bullet_{(k-1)}$).

\begin{lem*}
All elements of $\Fc_k$ satisfy the cocycle condition.
\end{lem*}

Recall that $\Fc_k$ identifies with a subspace of $E^k$ consisting of all $\xi\in E^k$ with $D\xi=0$,
 equivalently, $d\xi=0$.
In the oriented case, this means that $\xi$ is constant on every flag $\Lc\in\Lambda_k$,
 that is, $\xi[e,\Lc]$ depends only on $\Lc$ and on the orientation of $e$, not on $e$ itself.

\begin{proof}
Consider any $\xi\in E^k$ such that $D\xi=0$, and a generator $[e,\Lc]$ of $C^{k+1}_{(k-1)}$, where
 $\Lc\in\Lambda_{k-1}$ and $e$ is an oriented dual $(k+1)$-cell such that $e\cap [\Lc]\ne \0$.
Note that $e\cap [\Lc]$ is then a 2-cell of $[\Lc]_\star$.
We have
$$
d\iota^{*}\xi[e,\Lc]=\sum_{\Lc'} \iota^{*}\xi \left(\d^{\Lc'}[e,\Lc']\right),
$$
 where the sum is over all $\Lc'\in\Lambda_k$ such that $\Lc\subset\Lc'$ and $e\cap [\Lc']\ne \0$,
 that is, $\Lc'=\Lc\cup\{M\}$ for some codimension $k$ flat $M\subset [\Lc]$ that
 contains the center of $e$.
The center of $e$ is contained in a unique codimension $k+1$ flat $L$,
 and $\iota^{*}\xi \left(\d^{\Lc'}[e,\Lc']\right)$ equals $D\xi[e,\Lc'']$, where $\Lc'':=\Lc'\cup\{L\}$.
Since $D\xi=0$, the cocycle condition follows.
\end{proof}

\subsection{Embedding theorem}
\label{ss:emb-thm}
\emph{The map $f\mapsto D^k f$ induces an isomorphic embedding of $\Vc_k$ into $\Fc_k$.}

\medskip

Using this embedding, we will further on identify $\Vc_k$ with the corresponding subspace of $\Fc_k$.

\begin{proof}
By definition, $D^k f=0$ means that $f\in\Vc_{< k}$.
Therefore, $f\mapsto D^k f$ embeds the quotient space $\Vc_k=\Vc_{\le k}/\Vc_{< k}$ into $E^k$.
Moreover, the image $D^k f$ lies in $\Fc_k\subset E^k$, by \ref{ss:Fck}.
\end{proof}

The vector space $\Fc_k$ has an explicit basis, and we can obtain
 an explicit description of the quotients $\Vc_k$ if we are able to
 give explicit system of linear equations defining $\Vc_k$ as a subspace of $\Fc_k$.

\section{Relations}
\label{s:rels}
Earlier, we defined a natural embedding of $\Vc_k$ into $\Fc_k$.
We often identify $\Vc_k$ with the image of this embedding.
Our next goal is describing the linear relations that define $\Vc_k$ as a vector subspace of $\Fc_k$.

\subsection{Reciprocity laws}
\label{ss:recip}
Reciprocity laws for elements of $E^k$, in the case of general topological arrangements,
 look the same as in \ref{ss:recipEk}.
In terms of the cochain spaces $E^k$, they can be stated as follows.
Consider a partial flag $\Lc^\circ$ that includes flats of all codimensions from $0$ to $k$
 except just one codimension $m$ with $0<m<k$.
A cochain $\xi\in E^k$ satisfies \emph{the reciprocity law at} $(\Lc^\circ,e)$, where
  $e$ is a dual oriented $k$-cell with $e\cap [\Lc^\circ]\ne\0$, if
$$
\sum_{M} \xi[e,\Lc^\circ\cup\{M\}]=0.
$$
Here, the sum is over all flats $M$ such that $\Lc^\circ\cup\{M\}\in\Lambda_k$.
The pair $(m,k)$ is called the \emph{type} of the corresponding reciprocity law.

\label{ss:intrec}
Below, we give an interpretation of type $(k-1,k)$ reciprocity laws, for $k>1$, in terms of the
 coboundary operator $d$.

\begin{lem*}
An element $D\eta\in E^k$ satisfies all type $(k-1,k)$ reciprocity laws
 if and only if $d\iota^{*}\eta=0$.
\end{lem*}

Here, the isomorphism $\iota^{*}$ (see \ref{p:iso}) takes $\eta$ to an element of $C^{k-1}_{(k-2)}$,
 and then the coboundary operator $d=d_{(k-2)}$ of the complex $C^\bullet_{(k-2)}$ takes
 $\iota^{*}\eta$ to an element of $C^{k}_{(k-2)}$.

\begin{proof}
Consider a cochain $\xi\in E^k$ of the form $\xi=D\eta$ for some $\eta\in E^{k-1}$.
By definition, the value of $\xi$ on a generator $[e,\Lc]$ of $E_k$, where $\Lc\in\Lambda_k$
 and $e$ is an oriented dual $k$-cell with $e\cap [\Lc]\ne \0$, is equal to
$$
\xi[e,\Lc]=\eta_{\Gamma\Lc}(\d^{\Gamma\Lc} e).
$$
 where $\eta_{\Gamma\Lc}\in C^{k-1}_{\Gamma\Lc}$ is the $\Gamma\Lc$-component of $\eta$.
Suppose now that $\Lc=\Lc^\circ\cup\{M\}\in\Lambda_k$, where $\Lc^\circ=\Lc\sm\{M\}$ is fixed,
 and a codimension $k-1$ flat $M$ varies. 
Take the sum over $M$ in both parts of the above equality.
In the left-hand side, we obtain the same expression as for the reciprocity law at $\Lc^\circ$.
On the other hand, in the right-hand side, the sum of $\d^{\Gamma\Lc}$ over all $M$
 acts as $\d^{\Gamma\Lc^\circ}$: indeed, every dual $(k-1)$-cell on the boundary of $e$
 intersecting $[\Gamma\Lc^\circ]$ also intersects exactly one $M$.
We obtain, therefore, $d\iota^{*}\eta[e,\Gamma\Lc^\circ]$ in the right-hand side.
\end{proof}

\subsection{Emergence}
\label{p:typek-1k}
Here is how the reciprocity laws emerge from the action of $D^2$.

\begin{lem*}
For any $k>1$ and any $\zeta\in E^{k-2}$,
 the cochain $D^2\zeta\in E^k$ satisfies the reciprocity laws of type $(k-1,k)$.
In particular, all elements of $\Vc_k\subset\Fc_k$ satisfy the reciprocity laws of type $(k-1,k)$ for $k>1$.
\end{lem*}

\begin{proof}
Set $\xi:=D^2\zeta$ and $\eta:=D\zeta$; then $\xi=D\eta$.
It follows from \ref{ss:intrec} that type $(k-1,k)$ reciprocity laws for $\xi$ take the form $d\iota^{*}\eta=0$.
Recall that $D=\iota^{*-1}d$, therefore,
$$
d\iota^{*}\eta=d\iota^{*}\iota^{*-1}d\zeta=dd\zeta=0.
$$
Hence, indeed, $\xi$ satisfies the type $(k-1,k)$ reciprocity laws.
\end{proof}

\subsection{Heritage}
\label{t:recip}
Reciprocity laws are inherited by the images of $D$. 

\begin{lem*}
If $\xi\in E^k$ satisfies type $(m,k)$ reciprocity laws for a given $m$ such that $0<m<k$, then $D\xi$
 satisfies type $(m,k+1)$ reciprocity laws.
In particular, all elements of $\Vc_k\subset\Fc_k$ for $k>1$ satisfy the reciprocity laws of all types $(m,k)$ with $0<m<k$.
\end{lem*}

\begin{proof}
Assume that $\xi\in E^k$ satisfies all type $(m,k)$ reciprocity laws for a given $m$ such that $0<m<k$, and consider $D\xi$.
Choose any flag $\Lc^\circ$ such that $\Lc^\circ\cup\{M\}\in\Lambda_{k+1}$ for some codimension $m$ flat $M$,
 and consider the left-hand side in the type $(m,k+1)$ reciprocity law for $D\xi$ at $(e,\Lc^\circ)$:
$$
\sum_M D\xi[e,\Lc^\circ\cup\{M\}]=\sum_M \xi[e^+,\Gamma\Lc^\circ\cup\{M\}] -
 \sum_M \xi[e^-,\Gamma\Lc^\circ\cup\{M\}]
$$
Here, $e^\pm$ are defined as in \ref{ss:coch}, and the summation is over all $m$-flats
 $M$ such that $\Lc^\circ\cup\{M\}\in\Lambda_{k+1}$ or, equivalently, such that
 $\Gamma\Lc^\circ\cup\{M\}\in\Lambda_k$.
Note also that $\Gamma\Lc^\circ\cup\{M\}=\Gamma(\Lc^\circ\cup\{M\})$ since $m<k$.
In the right-hand side of the last formula, both terms vanish, by
 type $(m,k)$ reciprocity laws for $\xi$.

The last claim of the theorem follows from \ref{p:typek-1k} and from the first claim.
\end{proof}

\subsection{Period vanishing conditions}
\label{ss:pervan}
Let $\xi\in E^k$ satisfy the cocycle condition.
Consider the image $\omega:=\iota^*\xi\in C^k_{(k-1)}$ of $\xi$ under the isomorphism $\iota^*$.
Since $\xi$ satisfies the cocycle condition, the cochain $\omega$ is a cocycle: $d\omega=0$.
We may want to further require that $\omega$ be a coboundary, that is, that $\omega$ represent
 the trivial cohomology class.
More precisely, consider any flag $\Lc\in\Lambda_{k-1}$.
The $\Lc$-component $\omega_\Lc$ of $\omega$ is a cocycle in $C^k_\Lc$ or,
 equivalently, a 1-cocycle in $C^1([\Lc];\Oc_{[\Lc]})$.
Consider also a set of cycles $\al_1$, $\dots$, $\al_m$ in $C_1([\Lc];\Oc_{[\Lc]})$ whose
 homology classes generate the vector space $\Hs_1([\Lc];\Oc_{[\Lc]})$.
Imposing that the cohomology class of $\omega_\Lc$ is trivial is then equivalent to
  imposing that the \emph{periods} $\<\omega_\Lc,\al_i\>$ vanish for all $i=1$, $\dots$, $m$.
For this reason, this condition --- that $\omega_\Lc$ is a coboundary in $C^k_\Lc$ ---
 is called the \emph{period vanishing condition} at $\Lc$.
The following lemma is immediate from the definitions.

\begin{lem*}
If $\xi\in E^k$ satisfies the cocycle condition and all the period vanishing conditions, then $\xi=D\eta$
 for some $\eta\in E^{k-1}$.
\end{lem*}

Suppose now that $k=n=\dim(X)$.
In this case, cocycle conditions are vacuous (they hold automatically).
Note that an element $\xi\in E^n$ satisfies the period vanishing conditions if and only if,
 for every flag $\Lc^\circ\in\Lambda_{n-1}$, one has
$$
\sum_{a\in [\Lc^\circ]} \xi[e_a,\Lc^\circ\cup\{a\}]=0,
$$
 where the sum is over all the vertices (=dimension 0 flats) $a$ of $\Ac$ such that $a\in [\Lc^\circ]$, equivalently,
 such that $\Lc^\circ\cup\{a\}\in\Lambda_n$, and $e_a$ is the dual $n$-cell containing $a$.
The orientations of all $e_a$ are assumed to be consistent; if they cannot be chosen consistently,
 then the corresponding reciprocity law is vacuous.
More precisely, the meaning of consistent orientations is that $\sum_a [e_a,\Lc^\circ]\in C^{\Lc^\circ}_{n}$ is a cycle.
Note that the only period of $(\iota^*\xi)_{\Lc^\circ}$ is its value on the corresponding 1-cycle in $C_1([\Lc^\circ],\Oc_{[\Lc^\circ]})$,
 therefore, indeed, the formula displayed above is equivalent to the period vanishing conditions for $\xi$.

Clearly, period vanishing conditions as defined here specialize to
 period vanishing conditions of \ref{ss:pervanEk} in the case when $\Ac$
 is an arrangement of rational hyperplanes in the torus.

\subsection{Choice functions}
\label{ss:choice}
\def\Is{\mathsf{I}}
A \emph{choice function of level $k$} is by definition a map $\phi:\Lambda_k\to X$
 with the property that $\phi(\Lc)$ is a dual 0-cell of $[\Lc]$, that is,
 the center of some open (top dimension) cell in $[\Lc]$.
In constructions that follow, 
 we fix a choice function $\phi_k:\Lambda_k\to X$ for each $k$
 and denote all these choice functions by the same symbol $\phi$.

\subsection{Integration}
\label{ss:integr}
Suppose now that $\Lc\in\Lambda_k$ and that $\xi\in C^{k+1}_\Lc$ is a coboundary, and we want to define
 some specific $\eta\in C^k_\Lc$ with $d\eta=\xi$.
Let $a$ be an arbitrary dual $0$-cell of $[\Lc]$, and let $e$ be an oriented $k$-cell such that $e\cap [\Lc]=\{a\}$.
Recall that the flat $[\Lc]$ is connected; therefore, there is a path from $\phi(\Lc)$ to $a$
 represented by a finite sequence
$$
a_0=\phi(\Lc),\  I_1,\  a_1,\  \dots, a_{m-1},\  I_{m},\  a_m=a,
$$
 where $I_i$ for $i=1$, $\dots$, $m$ are dual 1-cells of $[\Lc]$, and $a_i$ for $0<i<m$ is the common
 endpoint of $I_{i}$ and $I_{i+1}$, which is then a dual 0-cell of $[\Lc]$.
For every $i$, choose an oriented dual $k$-cell $e_i$ of $X$ such that $e_i\cap [\Lc]=\{a_i\}$.
Orientations of $e_i$ can be chosen so that $e_m=e$ (taking orientations into account)
 and so that to satisfy the following requirement:
 for $0<i<m$, the dual cells $e_i$ and $e_{i+1}$ are oriented as the boundary of the same
 oriented dual $(k+1)$-cell $e_{i,i+1}$, whose intersection with $[\Lc]$ equals $I_i$.
Define
$$
\eta(e):=\sum_{i=1}^{m-1} \xi[e_{i,i+1},\Lc].
$$

We know that $\xi$ is a coboundary, therefore, its value on any cycle is 0.
It follows that $\eta$ is well defined, i.e., the value $\eta(e)$ given by the above formula
 does not depend on the choice of a path from $\phi(\Lc)$ to $a$.
Note also that, once $e$ and the path have been chosen, the dual cells $e_{i,i+1}$, including their orientations,
 are recovered uniquely.
As can be verified by a straightforward computation, $d\eta=\xi$ --- this is, in fact,
 a discrete version of a standard ``integration'' procedure for exact 1-forms going back to Poincar\'{e}.
The cochain $\eta$ as defined above is denoted by $\Is_\phi(\xi)$.
Call $\Is_\phi$ the \emph{integration map} corresponding to the choice function $\phi$.
Properties of the integration map are summarized in the following lemma.

\begin{lem*}
For a fixed $\phi$, the map $\Is_\phi:C^{k+1}_\Lc\to C^k_\Lc$ is a linear map with the property
 $d \Is_\phi=id$ on $dC^k_\Lc$.
\end{lem*}

Taking the direct sum over all flags $\Lc\in\Lambda_k$, we obtain a map from $C^{k+1}_{(k)}$ to $E^k$,
 which will also be denoted by $\Is_\phi$.
As follows from the lemma, $\Is_\phi$ is a partial right inverse of $d$ defined on $d E^k$,
 that is $d\Is_\phi\xi=\xi$ for every $\xi\in d E^k$.

Note that the complex $C^\bullet_\Lc$, as well as the integration map $\Is_\phi:C^{k+1}_\Lc\to C^k_\Lc$
 depend only on the base flat $[\Lc]$ and its coorientation, not on other flats in $\Lc$.
For this reason, we may also use the above lemma in the context where $\Lc$ is a partial flag not in $\Lambda_k$.

\subsection{Properties of the integrals}
\label{ss:intrec}
The following lemma relates properties of $\xi\in E^k$ such as reciprocity laws to properties of
 $\Is_\phi\xi$.

\begin{lem*}
Suppose that $\xi\in E^k$ satisfies the cocycle condition and all the period vanishing conditions.
Then $\eta:=\Is_\phi\iota^{*}\xi$ is well defined, and $\xi=D\eta$.
Furthermore, if $\xi$ satisfies the reciprocity laws of all types $(m,k)$ with $0<m<k$,
 then $\eta$ satisfies the reciprocity laws of all types $(m,k-1)$ with $m<k-1$
 as well as the cocycle condition.
\end{lem*}

\begin{proof}
Assume that $\xi\in E^k$ satisfies the cocycle condition and all the period vanishing conditions.
By \ref{ss:pervan}, one has $\xi\in DE^{k-1}$, equivalently, $\iota^{*}\xi\in dC^{k-1}_{(k)}$.
It follows that $\eta:=\Is_\phi\iota^{*}\xi$ is well defined, moreover,
$$
D\eta=\iota^{*-1}d\eta=\iota^{*-1}\iota^{*}\xi=\xi.
$$

Now suppose additionally that $\xi$ satisfies the reciprocity laws of all types $(m,k)$ with $0<m<k$.
Fix $m<k-1$ and write down a reciprocity law for $\xi$ at, say, $[e,\Lc^\circ]$,
 namely, $\sum_M \xi[e,\Lc^\circ\cup\{M\}]=0$,
 where the sum is over all dimension $m$ flats $M$ such that $\Lc^\circ\cup\{M\}\in\Lambda_{k}$.
Applying the isomorphism $\iota^{*}$ to both sides, we obtain that
$$
\sum_M \iota^{*}\xi[e,\Gamma\Lc_M]=0,\quad \Lc_M:=\Lc^\circ\cup\{M\}.
$$
Note that $\iota^*\xi\in C^k_{\Gamma\Lc^\circ}$ and that $\Gamma\Lc_M\in\Lambda_{k-1}$.
The integration operator $\Is_\phi$ is well defined on $C^{k}_{\Gamma\Lc^\circ}$
 (even though $\Lc^\circ$ is an incomplete flag that misses a codimension $m$ flat;
 see the end of \ref{ss:integr})
 and takes this space to the space $C^{k-1}_{\Gamma\Lc^\circ}$.
It follows that, acting by $\Is_\phi$ on both sides of the last equality,
 we obtain the reciprocity laws for $\eta=\Is_\phi\iota^*\xi\in C^{k-1}_{\Lc^\circ}$
 at all $[e';\Gamma\Lc^\circ]$, where $e'$
 is an oriented dual $(k-1)$-cell of $X$ such that $e'\cap[\Gamma\Lc^\circ]\ne\0$.
Hence, $\eta$ satisfies the reciprocity laws of all types $(m,k-1)$, where $0<m<k-1$.

The situation is slightly different for $m=k-1$, that is, when we try to prove the cocycle condition for $\eta$.
Even though one obtains in the same way as above that
 $\sum_M \iota^{*}\xi[e,\Gamma\Lc_M]=0$, the base flat of $\Gamma\Lc_M$
 is now $M$, and it is different for different terms in the sum.
On the other hand, using that $\iota^{*}\xi=d\eta$ and using the definition of $d$,
 we obtain that
$$
\sum_M \eta[e^+_M,\Gamma\Lc_M]-\eta[e^-_M,\Gamma\Lc_M]=0,
$$
 where $e^\pm_M$ are constructed by $e$ as $e^\pm$ from \ref{ss:coch}; they depend on $M$.
By definition of the isomorphism $\iota$, the term $\eta[e^\pm_M,\Gamma\Lc_M]$
 can be replaced with $\iota^{*}\eta[e^\pm_M,\Gamma\Lc^\circ]$.
Recall that the base flat of $\Gamma\Lc^\circ$ has codimension $k-2$.
Finally, the sum over $M$ being equal to zero is interpreted as $d\iota^{*}\eta[e,\Gamma\Lc^\circ]=0$
 (indeed, every dual $(k-1)$-cell in the boundary of $e$, oriented as in $\d e$
 and crossing the flat $[\Gamma\Lc^\circ]$,
 can be uniquely identified as $e^\pm_M$, for a flat $M$ of dimension $m=k-1$
 such that $\Lc_M\in\Lambda_k$).
As $e$ is arbitrary, we obtain that $\eta$ satisfies the cocycle condition.
\end{proof}

\section{Affine and pseudoaffine arrangements}
\label{s:aff}
An arrangement of hyperplanes in an affine space is not a topological arrangement,
 since the affine space is not compact.
However, it yields a topological arrangement in a straightforward fashion.
One only needs to pass to the one point compactification $\Sb^n=\R^n\cup\{\infty\}$
 of the affine space $\R^n$.
Here $\Sb^n$ is, of course, homeomorphic to the $n$-sphere.
In this way, every arrangement of affine hyperplanes in $\R^n$ defines
 a topological arrangement in $\Sb^n$, in which all hyperplanes pass through the point $\infty$.
This is a special case of pseudoaffine topological arrangements defined below.

\subsection{Pseudoaffine topological arrangements}
A topological arrangement $(X,\Ac)$ is said to be \emph{pseudoaffine} if the following conditions hold:
\begin{enumerate}
  \item every flat of $\Ac$ of dimension $>0$ is homeomorphic to a sphere;
  \item there is a special point $\infty\in X$ that lies in all hyperplanes. 
\end{enumerate}
Recall that $X$ is also a flat, hence it is also assumed to be a topological sphere.
On the other hand, we do not require that the intersection of two or several
 hyperplanes is connected; it may have several connected components but
 each must be a topological sphere (unless it is zero dimensional).
When discussing $(X,\Ac)$, we will use notation and terminology introduced above for
 general topological arrangements.
Note that a pseudoaffine arrangement is always orientable.

\subsection{Description of $\Vc_k$ for $k<n$}
\label{ss:desc-Vck}
The quotients $\Vc_k$ of the Varchenko--Gelfand algebra can be described as
 vector subspaces of $\Fc_k$.
Recall also that $\Fc_k$ is the space of all alternating functions on $\Lambda^\Oc_k$.
Below, we assume that $k<n:=\dim(X)$.

\begin{thm*}
Elements of $\Vc_k\subset\Fc_k$ for $k<n$ are precisely those alternating functions on $\Lambda^\Oc_k$
 which satisfy all the reciprocity laws.
\end{thm*}

\begin{proof}
We already know that all elements of $\Vc_k$ satisfy all the reciprocity laws, see \ref{t:recip}.
Now let $\xi\in\Fc_k$ satisfy the reciprocity laws of all types $(m,k)$ with $0<m<k$.
By \ref{ss:cocycle}, the cochain $\xi$ also satisfies the cocycle condition.
Also, since $\Hs_1([\Lc];\Oc_{[\Lc]})=0$ for every $\Lc\in\Lambda_{k-1}$ (recall that $k<n$,
 so that $[\Lc]$ is a topological sphere of dimension at least 2),
 the period vanishing conditions hold automatically.
It follows that $\xi=D\eta$ for a well defined $\eta\in E^{k-1}$.
Moreover, by \ref{ss:intrec}, the reciprocity laws for $\xi$ imply both the cocycle condition and
 the reciprocity laws for $\eta$ of types $(m,k-1)$ with $0<m<k-1$.
Again, the period vanishing conditions for $\eta$ are automatic, and we can iterate the argument.
We conclude that $\xi=D^kf$ for some $f\in\Vc_{\le k}$.
\end{proof}

\subsection{First description of $\Vc_n$}
\label{ss:1stdesc}
In order to describe $\Vc_n$ as a vector subspace of $\Fc_n$, we need to use
 the period vanishing conditions.

\begin{thm*}
The subspace $\Vc_n\subset\Fc_n$ consists exactly of those alternating functions on $\Lambda^\Oc_n$
 that satisfy the reciprocity laws as well as the period vanishing conditions.
\end{thm*}

\begin{proof}
Elements of $\Vc_n$ satisfy all the reciprocity laws and the period vanishing conditions,
 as follows from \ref{t:recip} and \ref{ss:pervan}.
On the other hand, suppose that $\xi\in\Vc_n$ satisfies all the reciprocity laws and the period vanishing conditions.
By \ref{ss:pervan}, there exists $\eta\in E^{n-1}$ such that $\xi=D\eta$.
Moreover, by \ref{ss:intrec}, one can choose $\eta$ so that it satisfies the cocycle condition as
 well as the reciprocity laws of all types $(k,n-1)$ with $0<k<n-1$.
The rest of the proof is now the same as in \ref{ss:desc-Vck};
 one simply performs further integrations until one obtains $f\in E^0$ with $D^n f=\xi$.
\end{proof}

The pseudoaffine case is somewhat special in that the period vanishing
 conditions are only needed for $\Vc_n$, not for $\Vc_k$ with $k<n$.
On the other hand, there is a description of $\Vc_n$ that does not
 refer to period vanishing conditions, as explained below.

\subsection{Second description of $\Vc_n$}
\label{ss:2nddesc}
\def\cb{\mathbf{c}}
An alternative way of describing $\Vc_n$ is as follows.
Let $\Fc^{\mathbf{c}}_n$ (here $\mathbf{c}$ is from ``$\mathbf{c}$ompact'')
 be the space of all alternating functions on $\Lambda^\cb_n$,
 where $\Lambda^\cb_n$ consists of all oriented flags $\Lc\in\Lambda_n$ except those with $[\Lc]=\{\infty\}$,
 that is, we consider only flags based at finite points.
Clearly, the dimension of $\Fc^\cb_n$ over $\Q$ is precisely the number of flags $\Lc\in\Lambda_n$ with $[\Lc]\ne\{\infty\}$;
 call such flags \emph{finite flags}.
There is a natural linear homomorphism $\rho^\cb:\Fc_n\to\Fc^\cb_n$ given
 by the restriction of a function on $\Lambda^\Oc_n$ to $\Lambda^\cb_n$.
Note that $\rho^\cb$ is surjective but not injective.
On the other hand, $\rho^\cb$ is injective on $\Vc_n\subset \Fc_n$, since
 the value of $\xi\in\Vc_n$ on a flag $\Lc\in\Lambda_n$ with $[\Lc]=\{\infty\}$
 can be recovered from the values $\xi[\Gamma\Lc\cup\{a\}]$, where $a$ ranges
 through the 0-cells of $\Ac$ in $[\Gamma\Lc]$, using the period vanishing condition at $\Gamma\Lc$.
It therefore suffices to describe $\rho^\cb\Vc_n$ as a subset of $\Fc^\cb_n$.

\begin{thm*}
The subspace $\rho^\cb\Vc_n\subset\Fc^\cb_n$ consists of all the alternating
 functions on $\Lambda^\cb_n$ that satisfy all the reciprocity laws.
\end{thm*}

Even though, formally, reciprocity laws were defined for elements of $\Fc_n$
 (or, more generally, of $E^n$), they clearly make sense in $\Fc^\cb_n$.

\begin{proof}
Suppose that $\xi^\cb\in\Fc^\cb_n$ satisfies all the reciprocity laws.
By the above, there is a unique $\xi\in\Fc_n$ such that $\rho^\cb\xi=\xi^\cb$
 and such that $\xi$ satisfies the period vanishing conditions.
The reciprocity laws for $\xi$ are automatic from those for $\xi^\cb$
 (including those at flags $\Lc\ni\{\infty\}$; they are obtained by summing
 up the other reciprocity laws).
By \ref{ss:1stdesc}, one has $\xi\in\Vc_n$, hence $\xi^\cb\in\rho^\cb\Vc_n$.
\end{proof}

\subsection{Valuations on polyhedra}
\label{ss:val-polyh}
We now impose the restriction that $\Ac$ consists of affine hyperplanes in $\R^n$
 but drop the assumption that $\Ac$ is finite.
Our objective in Section \ref{ss:val-polyh} is to describe a space of
 simple valuations on $\Ac$-polyhedra, that is, on all, not necessarily bounded,
 convex polyhedra in $\R^n$ whose facets lie in hyperplanes from $\Ac$.
We restrict our attention to \emph{cosimple} simple valuations
 on $\Ac$-polyhedra: those simple valuations $\mu$, which satisfy $\mu(P)=0$
 whenever $P$ contains an affine subspace of dimension $>0$.
Let $\Vc$ be the space of all convex chains with respect to $\Ac$,
 that is, the space of all functions 
 on $\R^n$ that are
 locally constant outside of a finite union of hyperplanes from $\Ac$.
These functions are viewed as elements of $L^\infty(\R^n,\Q)$;
 in particular, two such functions are identified if they differ only on a set of measure zero.
The vector space $\Ac$ is then spanned by the indicator functions of (not necessarily bounded) convex $\Ac$-polyhedra.
Spaces $E^k$, the Leray operators $D:E^k\to E^{k+1}$, the degree filtration $(\Vc_{\le k})_k$,
 and the corresponding quotients $\Vc_k$. 
 are defined verbatim as in \ref{ss:elem} -- \ref{ss:degfilt0}.
Also, for each $k$, consider the space $\Fc^\cb_k$ of all alternating functions on oriented flags from $\Ac$;
 we use the superscript $\Fc^\cb_k$ to emphasize that our flags are finite (not include $\infty$).
The space $\Vc_n$ embeds naturally into $\Fc^\cb_n$.
Note that cosimple simple valuations on $\Ac$-polyhedra identify with linear
 functionals on $\Vc_n$; the space of such valuations is therefore $\Vc^*_n$.
Firstly, we describe $\Vc_n$ as vector subspace of $\Fc^\cb_n$.

\begin{thm*}
An element of $\Fc^\cb_n$ lies in $\Vc_n$ if and only if it satisfies all the reciprocity laws.
\end{thm*}

\begin{proof}
Since an element of $\Fc^\cb_n$ has finite support, it suffices to assume that $\Ac$ is finite;
 in this case, the statement is identical with \ref{ss:2nddesc}.
\end{proof}

Similarly to the setup of the main theorem, methods of Section \ref{s:linalg} now imply
 the following description of $\Vc^*_n$.

\begin{thm*}
The space of cosimple simple valuations on $\Ac$-polyhedra
 is generated by Hadwiger invariants $\Upsilon_\Lc$ associated with oriented complete flags $\Lc\in\Lambda^\Oc_n$.
Relations between these invariants are generated by the reciprocity laws.
\end{thm*}

Here, as before, infinite linear combinations are allowed both for generators and for relations.

\subsection{Valuations on polytopes}
\label{ss:val-polyt}
Let $\Ac$ be as in \ref{ss:val-polyh} but now consider simple valuations on $\Ac$-polytopes (=bounded $\Ac$-polyhedra).
Keeping the same notation as in \ref{ss:val-polyh}, denote by $\Vc_\cb$ the
 subspace of $\Vc$ spanned by the indicator functions of $\Ac$-polytopes.
Equivalently, $\Vc_\cb$ consists of all compactly supported functions from $\Vc$,
 more precisely, of functions with compact essential support (disregarding possibly
 nonzero values on measure zero sets).

\begin{lem*}
The space $\Vc_\cb\cap\Vc_{<n}$ is zero, therefore, $\Vc_\cb$ embeds into $\Vc_n$.
\end{lem*}

\begin{proof}
Assume, by way of contradiction, that $\al\in\Vc_\cb\cap\Vc_{<n}$ is nonzero.
We may also assume that $\Ac$ is finite, since only finitely many
 hyperplanes from $\Ac$ are needed to define $\al$.
It suffices to show that $D^n\al\ne 0$ to obtain a contradiction with $\al\in\Vc_{<n}$.
Represent $\al$ by a function (denoted by the same symbol) with $\{q\in\R^n\mid \al(q)\ne 0\}$ open.
Then $\supp(\al)$ is a finite union of convex $n$-polytopes.
Define a complete flag $\Lc=(L_0\supset \dots\supset L_n)$
 and a nested sequence of (possibly, non-convex) polytopes $P_0\supset \dots\supset P_n$ inductively by $k$.
Set $L_0=\R^n$ and $P_0=\supp(\al)$;
 for $k>0$, let $P_k$ be a facet of $P_{k-1}$ and let $L_k$ be the affine hull of $P_k$.
For the complete flag $\Lc$ thus defined, $D^n\al$ takes a nonzero value at $[e,\Lc]$,
 where $e$ the dual $n$-cell of $\Ac$ centered at the point $L_n$.
A contradiction with $\al\in\Vc_{<n}$.
\end{proof}

Since the vector space $\Vc_\cb$ naturally lies in $\Vc_n$, it can also be viewed as a subspace of $\Fc^\cb_n$.

\begin{thm*}
The subspace $\Vc_\cb\subset\Fc^\cb_n$ is given by the reciprocity laws and the
 period vanishing conditions.
\end{thm*}

\begin{proof}
Consider $\al\in\Fc^\cb_n$.
Since only finitely many hyperplanes are needed to define $\al$, we may assume that $\Ac$ is finite.
Let $X$ be the sphere $\Sb^n=\R^n\cup\{\infty\}$; then $\Ac$ defines a
 pseudoaffine topological arrangement in $X$ denoted by the same symbol $\Ac$.
Embed $\Fc^\cb_n$ into $\Fc_n$ by setting, for every $f\in\Fc^\cb_n$, the
 value of $f$ at any complete flag based at $\infty$ to zero.
Then $\Vc_\cb$ also embeds into $\Fc_n$.
By \ref{ss:1stdesc}, if $\al\in\Vc_\cb$, then $\al$ satisfies the reciprocity laws and the period vanishing conditions.
Assume now that $\al\in\Fc^\cb_n$ satisfies the reciprocity laws and the period vanishing conditions.
Integration map of \ref{ss:integr} applied $n$ times to $\al$ then yields an element of $\Vc$, by \ref{ss:1stdesc}.
Moreover, integration can be performed so that the latter element vanishes on a neighborhood of infinity
 (to this end, it suffices to take a choice function $\phi$ with the property that
 $\phi(\Lc)$ lies in an unbounded $\dim[\Lc]$-cell of $\Ac|_{[\Lc]}$, for every $\Lc$);
 it is then an element of $\Vc_\cb$ corresponding to $\al$.
\end{proof}

For the case where $\Ac$ consists of all affine hyperplanes in $\R^n$,
 the above result is equivalent to \cite[Theorem 3.5]{Dup01}.
Methods of Section \ref{s:linalg} now yield a description of $\Val(\Pc_\Ac;\mathsf{0})=\Vc^*_\cb$.

\begin{thm*}
The space $\Val(\Pc_\Ac;\mathsf{0})$
 of simple valuations on $\Ac$-polytopes
 is generated by Hadwiger invariants $\Upsilon_\Lc$ associated with oriented complete flags $\Lc\in\Lambda^\Oc_n$.
Relations between these invariants are generated by the reciprocity laws and the period vanishing conditions.
\end{thm*}

\section{Toric and pseudotoric arrangements}
\label{s:tor}
\def\ps{\mathsf{p}}
Recall that the integer lattice $\Z^n$ acts on the affine space $\R^n$ by 
 translations:
 given $v\in\Z^n$ and $q\in\R^n$, we write $q+v$ for the parallel translate of $q$ by the vector $v$.
A \emph{rational hyperplane} in $\R^n$ is by definition an affine hyperplane $\hat H$
 that coincides with the affine hull of $\Z^n\cap \hat H$, equivalently, is such
 that $\Z^n\cap \hat H$ is a lattice in $\hat H$.
Rational hyperplanes can also be characterized as those defined over $\Q$.

\subsection{Toric arrangements}
Let now $X$ be the real $n$-torus $\T^n=\R^n/\Z^n$.
Write $\ps:\R^n\to X$ for the quotient projection.
Define a \emph{rational hyperplane in $X$} as the $\ps$-image of a rational hyperplane in $\R^n$.
Note that all rational hyperplanes in $X$ pass through the same special point $o$
 that coincides with $\ps(v)$ for every $v\in\Z^n$.
A finite set $\Ac$ of rational hyperplanes in $X$ is a \emph{toric hyperplane arrangement} if,
 for every component $U$ of $X\sm\bigcup\Ac$, every component of $\ps^{-1}(U)$ is bounded
 (hence it is the relative interior of some convex polytope in $\R^n$).

Let $\hat\Ac$ be the set of all affine hyperplanes $\hat H\subset\R^n$ such that $\ps(\hat H)\in\Ac$.
Then $\hat\Ac$ is not finite but \emph{locally finite}, that is, every compact subset of $\R^n$
 intersects only finitely many hyperplanes from $\hat\Ac$.
As a rule, the hats over symbols will mean lifts to $\R^n$ of geometric objects from $X$ denoted by
 these symbols. 
Note that the minimal by inclusion $\hat\Ac$-polyhedra are precisely
 components of $\ps^{-1}(U)$, where $U$ ranges through the regions of $\Ac$.

\subsection{Pseudotoric arrangements}
\label{ss:pseudotoric}
The following notion generalizes that of toric arrangements.
Let $(X,\Ac)$ be a topological arrangement such that all flats of $\Ac$ are homeomorphic to compact tori $\R^m/\Z^m$.
In particular, since $X$ is a flat, it also identifies with a torus of dimension $n$.
Call $\Ac$ a \emph{pseudotoric arrangement}.
Even though the most interesting for us case is that of genuine toric arrangements,
 there are many pseudotoric arrangements that are not toric, even up to a homeomorphism.

From now on,
 we assume that $(X,\Ac)$ is pseudotoric.
All arguments are topological; for this reason, assuming that $\Ac$ is genuine toric will
 not yield any essential simplification.
Notation $\Vc_k$, $\Fc_k$, $E^k$, etc., introduced for general topological arrangements,
 is applicable to $X$.

\subsection{The space $\Vc_1$}
\label{ss:Vc1}
Recall that $\Fc_1$ is naturally a subspace of $E^1$, moreover,
 $\iota^{*}\xi\in C^1$ is automatically a cocycle, for every $\xi\in\Fc_1$; see \ref{ss:cocycle}.
Clearly, $\xi\in\Vc_1$ if and only if $\iota^{*}\xi$ is a coboundary,
 and the latter is equivalent to saying that $\xi$ satisfies the period vanishing conditions, see \ref{ss:pervan}.

\begin{lem*}
Elements of $\Fc_1$ represent all cohomology classes in $\Hs^1(X;\Q)$.
In other words, for any cocycle $\eta\in E^1$ there exists $\xi\in\Fc_1$ such that $\eta+\xi$
 satisfies the period vanishing conditions.
\end{lem*}

\begin{proof}
Firstly, consider a cooriented hyperplane $H$ from $\Ac$.
The cohomology class $[H]$ of $H$, that is, the Poincar\'e dual of the fundamental homology class of $H$,
 is represented by a function $\xi_H\in\Fc_1$ that is equal to $\pm 1$ on $H$
 (plus or minus depends on the coorientation) and to 0 on all other hyperplanes from $\Ac$.
It now suffices to show that the cohomology classes of hyperplanes from $\Ac$ generate $\Hs^1(X)$.
Dually, one needs to prove that a first homology class orthogonal to $[H]$ for all $H\in\Ac$ is necessarily trivial.
We now use that $X$ is a torus: the space $\Hs_1(X)$ is isomorphic to $\pi_1(X,b)\otimes\Q$,
 where $b$ is some base point that lies outside of $\bigcup\Ac$, say, one can take
 the center of any $n$-cell of $\Ac$.
Consider a PL loop $\ga$ based at $b$ that is transverse to all hyperplanes from $\Ac$
 and that has zero intersection index with any of these hyperplanes.
Our objective is to show that $\ga$ is then trivial.

Let $\tilde{\ga}$ be a lift of the path $\ga$ to the universal covering space $\R^n$ of $X$, that is,
 $\ps\circ\tilde\ga=\ga$.
Fix a hyperplane $H\in\Ac$ together with some coorientation of it.
The total number of intersection points between $\ga$ and $H$, counted with signs, is equal to zero.
It can be easily deduced that, for every component $\hat H$ of $\ps^{-1}(H)$, the endpoints of $\tilde\ga$
 are on the same side of $\hat H$.
Since this holds for every $H\in\Ac$,
 both endpoints of $\tilde{\ga}$ belong to the same minimal by inclusion $\hat\Ac$-polytope $\hat U$,
 which therefore projects onto a cell $U=\ps(\hat U)$ of $\Ac$.
But $\ps$ is one-to-one on $\hat U$, while both endpoints of $\tilde{\ga}$ map to the same point $b\in U$.
Hence, the two endpoints of $\tilde{\ga}$ coincide, and $\tilde{\ga}$ is a loop;
 the latter means that $\ga$ is trivial, as claimed.
\end{proof}

We can now compute the dimension of $\Vc_1$.

\begin{cor*}
If $N$ is the cardinality of $\Ac$, then $\dim\Vc_1=N-n$.
\end{cor*}

\begin{proof}
Consider the linear map from $\Fc_1$ onto $\Hs^1(X)$ taking each $\xi\in\Fc_1$
 to the cohomology class of the 1-cocycle $\iota^{*}\xi$; this map is 
 onto by the above theorem.
The kernel of this map coincides with $\Vc_1$, by \ref{ss:pervan}.
By general linear algebra, the dimension of the kernel equals the dimension $N$ of $\Fc_1$,
 minus the dimension $n$ of the image.
\end{proof}

In particular, we must have $N\ge n$, for every pseudotoric arrangement $\Ac$.
The reason is, any number $N<n$ of hyperplanes are not enough to separate $X$ into cells,
 i.e., are not enough to generate $\Hs^1(X)$.

\subsection{An abstract lemma}
\label{ss:abstract}
Consider a finite index set $\Lambda$ and a $\Lambda$-graded $\Q$-vector space $V:=\bigoplus_{\la\in\Lambda} V_\la$.
A linear relation of the form $\sum_{\la\in\Lambda} c_\la v_\la=0$ is called a \emph{$\Lambda$-relation}.
Here $v_\la\in V_\la$ are the components of a vector $v\in V$ satisfying the given relation,
 and $c_\la\in\Q$ are some coefficients.
Observe that it makes sense to talk about \emph{the same} $\Lambda$-relation in different
 $\Lambda$-graded vector spaces.
Now consider two $\Lambda$-graded vector spaces $U$ and $V$.
Define a \emph{$\Lambda$-diagonal linear map} as a linear map $f:U\to V$ such that $f(U_\la)\subset V_\la$,
 for every $\la\in\Lambda$.
It is easy to see that a $\Lambda$-diagonal linear map is surjective if
 and only if so is every restriction $f_\la:=f|_{U_\la}$.

\begin{lem*}
Let $f:U\to V$ be a $\Lambda$-diagonal linear map between $\Lambda$-graded vector spaces $U$ and $V$.
Choose $v\in f(U)$ satisfying a given set of $\Lambda$-relations.
Then there is a vector $u\in U$ satisfying the same $\Lambda$-relations and such that $f(u)=v$.
\end{lem*}

\begin{proof}
Firstly, the given set of $\Lambda$-relations can be solved with respect to a subset of $\Lambda$,
 that is, the given relations are equivalent to those expressing $u_{\la'}$ for $\la'\in\Lambda'$ through $u_{\la''}$ for $\la''\in\Lambda''$,
 where $\Lambda=\Lambda'\sqcup \Lambda''$ is a disjoint union representation.
In other words, if vectors $u_{\la''}\in U_{\la''}$ are chosen arbitrarily for all $\la''\in\Lambda''$,
 then the vectors $u_{\la'}$ for all $\la'\in\Lambda'$ can be uniquely recovered, using
 the given $\Lambda$-relations.
Secondly, for every $\la''\in\Lambda''$, choose $u_{\la''}$ so that $f(u_{\la''})=v_{\la''}$ and otherwise arbitrarily.
Recover $u_{\la'}$ for all $\la'\in\Lambda'$ using the given $\Lambda$-relations.
Since $f(u)$ satisfies the same $\Lambda$-relations, it follows that $f(u)_{\la'}=v_{\la'}$ for all $\la'\in\Lambda'$.
\end{proof}

The lemma just proved will be applied to the subspace of cocycles in $C^k_{(l)}$,
 which is a $\Lambda$-graded vector space with $\Lambda=\Lambda_l$, and to the $\Lambda$-diagonal
 linear map sending each cocycle to its cohomology class.

\subsection{Description of $\Vc_k$}
\label{ss:tork<n}
Go back to a pseudotoric arrangement $(X,\Ac)$.
Below, we describe the quotients $\Vc_k$ of the degree filtration.
Recall that $\Vc_k$ identifies with a vector subspace of $\Fc_k$, the space of all
 alternating functions on the set $\Lambda^\Oc_k$ of oriented flags.
In the following theorem, period vanishing conditions are understood in the sense of \ref{ss:pervan}
 (which is consistent with \ref{ss:pervanEk}).

\begin{thm*}
The subspace $\Vc_k\subset\Fc_k$ is formed by those alternating functions on $\Lambda^\Oc_k$
 that satisfy the reciprocity laws and the period vanishing conditions.
\end{thm*}

\begin{proof}
Given an element $\xi\in\Fc_k$ that satisfies the reciprocity laws 
 and the period vanishing conditions, we need to be able to ``integrate'' $\xi$
 as many times (namely, $k$) as needed to produce an element of $E^0$.
The first integration is possible due to the period vanishing conditions: there is some $\eta\in E^{k-1}$ with $\xi=D\eta$.
Moreover, one can assume that $\eta$ satisfies all the reciprocity laws of types $(m,k-1)$ with $m<k-1$,
 as well as the cocycle condition, by \ref{ss:intrec}.
A problem, however, is that $\eta$ does not need to satisfy the period vanishing conditions.

Consider $\zeta:=\iota^{*}\eta\in C^{k-1}_{(k-2)}$.
For any flag $\Lc\in\Lambda_{k-2}$, let $\zeta_\Lc$ be the $\Lc$-component of $\zeta$;
 it is a cocycle in $C^{k-1}_\Lc$ identified with a 1-cocycle in $C^1([\Lc];\Oc_{[\Lc]})$.
Write $[\zeta_\Lc]\in\Hs^1([\Lc];\Oc_{[\Lc]})$ for the cohomology class of $\zeta_\Lc$.
By \ref{ss:Vc1}, there is a 1-cocycle $\zeta'_\Lc$ constant on all dimension $k-1$
 flats in $[\Lc]$ and representing the same cohomology class $[\zeta_\Lc]$.
The element $\zeta'':=\bigoplus_\Lc (\zeta_\Lc-\zeta'_\Lc)$ now satisfies the period vanishing condition.
On the other hand, the reciprocity laws may fail for this element,
 for ``wrong'' choices of $\zeta'_\Lc$.
A ``right'' choice, satisfying the relevant reciprocity laws, can be made by \ref{ss:abstract}.
Setting $\eta':=\iota^{*-1}\zeta''$, we obtain an element $\eta'\in E^{k-1}$
 with $\xi=D\eta'$ such that $\eta'$ satisfies the period vanishing conditions and the
 reciprocity laws of all types $(m,k-1)$ with $0<m<k-1$.
Further integration steps are similar.

As a result of the integration process just described, one obtains an element $f\in\Vc$
 with $\xi=D^kf$, which means $\xi\in\Vc_k$, as desired.
\end{proof}

Note that an alternative description of $\Vc_n$ similar to \ref{ss:2nddesc}
 is possible in the pseudotoric case; this description considers
 all flags of $\Lambda_n$ except those with $\{o\}$ as the base flat,
 and it allows to go without the period vanishing conditions.
Since the period vanishing conditions are still needed for $k<n$, this alternative
 description is less useful than in the pseudoaffine case.

\subsection{Proof of the main theorem}
\label{ss:proof-main}
We now let $\Ac$ be an arbitrary collection of rational hyperplanes in $\T^n$
 containing at least one $n$-tuple of hyperplanes with zero-dimensional intersection.
Notation $\Vc_k$, $\Fc_k$, etc., now refers to this, possibly infinite, collection $\Ac$.
Since every element of $\Vc_k$ is defined by finitely many hyperplanes from $\Ac$,
 the description of the embedding $\Vc_k\subset \Fc_k$ looks the same as in \ref{ss:tork<n}.
In other words, we obtain the theorem of \ref{ss:descVck}.
By the results of Section \ref{s:linalg}, this implies the main theorem.


\end{document}